\documentclass[12pt,a4paper]{article}
\usepackage[utf8]{inputenc}
\usepackage[T1]{fontenc}
\usepackage{amsmath}
\usepackage{amsfonts}
\usepackage{amssymb}
\usepackage{makeidx}
\usepackage{graphicx}
\usepackage[T1]{fontenc}
\usepackage{csquotes}
\usepackage{hyperref}
\usepackage[french,english]{babel}
\usepackage{amsfonts,amsmath,amssymb}
\usepackage{amsmath,amscd,amsxtra,latexsym,mathrsfs}
\usepackage{amsthm}
\usepackage{bbold}
\usepackage{cases}
\usepackage{stackengine}
\usepackage{scalerel}
\usepackage{enumitem}
\usepackage{fullpage} 
\usepackage{graphicx}

\numberwithin{equation}{section}

\newtheorem{thm}{Theorem}[section]    
\newtheorem{prop}[thm]{Proposition}    
\newtheorem{defi}{Definition}[section]                             
\newtheorem{cor}[thm]{Corollary}    
\newtheorem{lemme}[thm]{Lemma}   
\newtheorem{NB}{\textbf{\underline{Remark}}}

\newcommand{\tg}{\tilde{g}}
\newcommand\hh{\mathbb{H}}
\newcommand\bbb{\mathbb{B}}
\newcommand{\pr}{\mathbb{P}}
\newcommand{\esp}{\mathbb{E}}

\newcommand{\symp}{\odot}

\def\bone{\mathbb 1}

\DeclareMathOperator{\Span}{Span}

\makeatletter
\def\oversortoftilde#1{\mathop{\vbox{\m@th\ialign{##\crcr\noalign{\kern3\p@}%
				\sortoftildefill\crcr\noalign{\kern3\p@\nointerlineskip}%
				$\hfil\displaystyle{#1}\hfil$\crcr}}}\limits}

\def\sortoftildefill{$\m@th \setbox\z@\hbox{$\braceld$}%
	\braceld\leaders\vrule \@height\ht\z@ \@depth\z@\hfill\braceru$}
\makeatother

\usepackage{authblk}

\newcommand{\Ge}{\mathbb G}
\newcommand{\He}{\mathbb H}
\newcommand{\R}{\mathbb R}
\newcommand{\N}{\mathbb N}

\newcommand{\W}{\mathcal W}

\def\cC{\mathcal C}

\newcommand{\E}{\mathbb E}
\newcommand{\ve}{\varepsilon}
\newcommand{\tr}{\mathrm{tr}}
\newcommand{\dil}{\mathrm{dil}}
\renewcommand{\P}{\mathbb P}

\newcommand{\trans}{\mathrm{trans}}
\newcommand{\rot}{\mathrm{rot}}

\newcommand{\SA}{{\mathscr A}}

\newcommand{\SD}{{\mathscr D}}

\newcommand{\SG}{{\mathscr G}}

\newcommand{\SM}{{\mathscr M}}
\newcommand{\SN}{{\mathscr N}}

\newcommand{\SSS}{{\mathscr S}}
\newcommand{\SU}{{\mathscr U}}
\newcommand{\SV}{{\mathscr V}}
\newcommand{\SW}{{\mathscr W}}

\def\dTV{d_{\textrm{TV}}}
\def\cN{\mathcal{N}}
\def\PP{\mathbb{P}}

\newcommand{\Hm}[1]{\leavevmode{\marginpar{\tiny%
$\hbox to 0mm{\hspace*{-0.5mm}$\leftarrow$\hss}%
\vcenter{\vrule depth 0.1mm height 0.1mm width \the\marginparwidth}%
\hbox to 0mm{\hss$\rightarrow$\hspace*{-0.5mm}}$\\\relax\raggedright
#1}}}

\usepackage{xcolor}

\title{A coupling strategy for {B}rownian motions at fixed time on {C}arnot groups using {L}egendre expansion}

\author{Marc Arnaudon}
\author{Magalie Bénéfice}
\author{Michel Bonnefont}
\author{Delphine Féral}
\affil{{Univ. Bordeaux, CNRS, Bordeaux INP, IMB, UMR 5251},{Talence},
           {F-33400}, 
            {France}}
\date{\today}

\begin{document}

\maketitle

\begin{abstract}
We propose a new simple  construction of a coupling at a fixed time of two sub-Riemannian Brownian motions on the Heisenberg group and  on the free step 2 Carnot groups. The construction is based on a  Legendre expansion of the standard Brownian motion and of the Lévy area. We deduce sharp estimates for the decay in total variation distance  between the laws of the Brownian motions. Using a change of probability method, we also obtain the log-Harnack inequality, a Bismut type integration by part formula and reverse Poincar\'e inequalities for the associated semi-group.
\end{abstract}

\section{Introduction}
\label{Section1}

Recently, the study of successful couplings for  Brownian motion  on sub-Riemannian manifolds has received a lot of attention. In the case of the  examples discussed below, the sub-Riemannian Brownian motion consists in a Riemannian Brownian motion on a base manifold together with its swept area.
The construction of successful couplings is thus a challenging question since one has to couple the Riemannian Brownian motions on the base manifold in such a way that also their swept area meet.
The first construction of successful couplings on the Heisenberg group or on the free step 2 Carnot groups were obtained by Ben Arous, Cranston and Kendall \cite{CranstonKolmogorov} and  Kendall \cite{kendall2007coupling,kendall-coupling-gnl}.

These  first couplings were Markovian couplings or at least co-adapted couplings.
A main progress was made by Banerjee, Gordina and Mariano in \cite{banerjee2017coupling} where they constructed a non co-adapted successful coupling on the Heisenberg group $\He$.   Their coupling is sometimes called   a finite look-ahead coupling since they repeat some  Brownian bridges couplings with   the use of the  future values of one  stochastic process. This kind of finite look-ahead coupling was already proposed by Banerjee and Kendall \cite{BanerjeeKolmogorov} in  some different hypoelliptic context: the Kolmogorov diffusion; i.e., a Brownian motion on $\R$ and its (iterated) time integral.

The second named author B\'en\'efice extended the co-adapted Kendall's coupling
  to the case of the curved sub-Riemannian manifold $SU(2)$  in \cite{KendallSU(2)} and the non co-adapted  coupling of Banerjee, Gordina and Mariano to the cases of $SU(2)$ and $SL(2,\R)$ in \cite{Nonco-adaptedSU(2)} and of higher dimensional Carnot groups in \cite{CarnotSuccessful}. Another interesting non co-adapted coupling on $\He, SU(2)$ and the universal covering of $SL(2,\R)$ was given recently  by Luo and Neel in \cite{luo2024nonmarkovian}.

Successful couplings are  interesting in themselves but have also a lot of analytical consequences  for the regularization of the associated semi-group and  for the study of the associated harmonic functions.

The construction of the finite look ahead coupling in \cite{banerjee2017coupling} is not so easy.
The main contribution of the present work is to  propose  a simpler construction for  the coupling of two sub-Riemannian Brownian motions starting from different points but only at a fixed time. We will consider the case  of the Heisenberg group  and its extension to the Carnot group case.   Our construction is based on a  Legendre expansion of the standard Brownian motion which, as it was noticed by Kuznetsov \cite{Kuznetsov2018}, is well adapted to the computation of the  L\'evy area, see Lemmas \ref{lem:Legendre-MB} and \ref{lem:Legendre-aire}.  
We will see that even if the coupling is only given for a fixed time and  thus is  not really a  successful coupling, we can still deduce some important regularization  properties for the associated semi-group.

A first direct application of successful couplings is total variation distance  estimates between the laws of two Markov processes. 
This comes from the Aldous inequality which  writes
\begin{equation}\label{E-intro:Aldous}
    d_{TV} \left(\mu_ t^{x}, \mu_t^{\tilde x} \right)  \leq \PP(X_t^x \neq X_t^{\tilde x})
\end{equation}
for any coupling $(X_t^x, X_t^{\tilde x})$ and 
with $\mu_t^x=\mathcal L (X_t^x)$  and $\mu_t^{\tilde x}=\mathcal L (X_t^{\tilde x})$.

For example, in the  case of the standard Brownian motion on $\R^n$, contrary to  $p$-Wasserstein distances with  $p\in[1,\infty]$, it permits to describe the  regularization of the standard heat semi-group with a (polynomial) decay. For $t>0$, $x,\tilde x\in \R^{n}$:
\begin{equation}
\label{eq:dTV-Rn}
\dTV \left(\mu_t^{x}, \mu_t^{\tilde x} \right)  \leq \frac{\Vert \tilde x-x\Vert }{\sqrt {2\pi t}} 
\end{equation} where $\mu_t^{x}=\mathcal N(x,t)$ is the law of the standard Brownian motion starting in $x$ on $\R^n$.
In fact, considering  the reflection coupling on $\R^n$, there is an equality in  the Aldous inequality \eqref{E-intro:Aldous} and we also have:
 \[
 \dTV \left(\mu_t^{x}, \mu_t^{\tilde x} \right)=\PP\left(\tau_{\frac{1} {2}|\tilde x -x |}>t\right)
 \]
with $\tau_{\frac{1} {2}|\tilde x -x |}$ the hitting time of $\frac{1} {2}|\tilde x -x |$ for a standard Brownian motion on $\R$ starting in 0.

Below, let us denote $\mu_t^{x_1,x_2,z}$ to be the law of the sub-elliptic Brownian motion on the Heisenberg group starting from $(x_1,x_2,z)$. The first main result of the present  paper is the extension of the total variation estimate  \eqref{eq:dTV-Rn} to the case of the Heisenberg group and to the case of the free step 2 Carnot groups. We state it below in Theorem \ref{T1} for the Heisenberg group.
The generalization to the case of the Carnot groups is given in Theorem \ref{T2}. The case of the Heisenberg group result already appears in \cite{banerjee2017coupling}. Another improvement here is that we obtain  explicit constants. 

\begin{thm}
\label{T1}
There exist two constants $C_1, C_2\geq 0$ such that for all $t\geq 0$
and all $(x_1,x_2,z)$ and $(\tilde x_1, \tilde x_2, \tilde z)$  in $\He$, 
\begin{equation}
\label{eq:main-H}
\dTV \left(\mu_t^{(x_1,x_2,z)}, \mu_t^{(\tilde x_1, \tilde x_2, \tilde z)} \right) \leq  C_1
\frac{\Vert(\tilde x_1- x_1, \tilde x_2-x_2)\Vert_2 }{\sqrt t} +  C_2 \frac{|\tilde z-z -\frac{1}{2} (x_1\tilde x_2- x_2 \tilde x_1 )|}{t}. 
\end{equation}
Moreover:
\[
C_2=\frac{5 \sqrt {21}}{\pi \sqrt{\pi}} 
\textrm{ and }
C_1=\frac{1}{\sqrt{2\pi}} + \sqrt{\frac{2}{3\pi}} C_2=\frac{1}{\sqrt{2\pi}} \left(1+ \frac{5\sqrt{28}}{\pi \sqrt \pi}\right).
\]
\end{thm}

As noticed in \cite{banerjee2017coupling}, Theorem \ref{T1} provides the sharp order of decay.  In this sense, the  associated  coupling is called \emph{efficient}. It is also noted in \cite{banerjee2017coupling} that any Markovian or co-adapted coupling  can not reach the sharp estimate when the initial point are in the same fiber, i.e., when $(\tilde x_1, \tilde x_2)= ( x_1, x_2)$.  The coupling proposed in \cite{luo2024nonmarkovian} is even  actually \emph{maximal} when the initial points are in the same fiber; i.e., similarly to the  reflection coupling in $\R^n$, it  produces an equality in the Aldous inequality \eqref{E-intro:Aldous}.

The second main type of  application of 
successful couplings are  gradient estimates for the associated semi-group 
and  for harmonic functions. A direct application of the total variation estimates  first leads to the following $L^\infty $ gradient bounds. It is stated here  for the Heisenberg group. The case of the free step $2$ Carnot groups will be given in Corollary  \ref{C2}.  
\begin{cor} \label{C-intro}
    For any bounded measurable  function   $f$ on $\He$, and any $t>0$:
	\begin{equation}\label{gradientInequality-h}
	\Vert \nabla_{\frak h }P_t f \Vert_\infty  \leq   \frac{2 C_1}{\sqrt{t}} ||f||_{\infty}
	\end{equation}
	and 
	\begin{equation}\label{gradientInequality-v}
	\Vert ZP_t f \Vert_\infty  \leq   \frac{2  \sqrt 2 C_2}{t} ||f||_{\infty}.
	\end{equation}
\end{cor}

In order to obtain stronger gradient inequalities, we may use a change of probability technique.
The idea is to construct couplings  with probability one   at a given fixed time of the two processes. The price to pay will be to make changes of probabilities for one of  the process. The distance between semigroups will be measured by this change of probability.
We first derive a log-Harnack inequality for the semi-group, see Theorem \ref{T-LH}. 
We then establish in Theorem \ref{T-B} a Bismut type formula; i.e., an integration by parts formula  for the derivative of the semi-group.
We deduce some reverse Poincar\'e  inequalities  for $p>1$,  see  Theorem \ref{T-RP} and a weak reverse log-Sobolev  inequality, see Corollary  \ref{C3}.
In a different hypoelliptic setting,  this change of probability  method was  investigated at least  by Guillin and  Wang  \cite{guillin-wang} and by Baudoin, Gordina and Mariano \cite{kolmogorov}  to study some  kinetic Fokker-Planck equation.

Another approach to obtain these gradient estimates is through the generalized curvature-dimension criterion developped by Baudoin and Garofalo \cite{BG-courbure}. Step 2 Carnot groups are examples of non-negatively curved sub-Riemannians manifolds with transverse symmetries and thus a reverse log-Sobolev is known to hold, see Proposition 3.1 in \cite{bb-logsob}.
See also \cite{bb-revpoinc} where the reverse Poincar\'e inequality  and its constant is studied on general Carnot groups
by analytic methods. A general stochastic method which also provides local estimates can be found in~\cite{ArnaudonThalmaier}, but the constants are not explicit.

In order to enlighten the simplicity of the method, we chose to present  first  the construction of the coupling and the total distance variation estimate  in the case of the Heisenberg group and 
to investigate only in a second time  the case of the  higher dimensional step 2 Carnot groups on $\R^n$, $n\geq 3$.  The reason is that  some small complication arises for the Carnot groups. The sub-Riemannian Brownian motion consists of $n$ independent 1-dimensional standard Brownian motions together  with all their $\frac{n (n-1)}{2}$ L\'evy areas. The main difference is that in this situation the vertical space is not anymore 1-dimensional. It is identified with $\frak{so}(n)$ and we have  used some Wishart matrices to get a solution of  Equation \eqref{E12}; see  Proposition \ref{P1}.  

The outline of the paper is the following.
In Section 2, we  quickly  describe the Heisenberg group and its sub-Riemannian Brownian motion. We then describe their nice expansion with the use of the orthogonal Legendre polynomials. Finally, we provide the proof of Theorem \ref{T1} for the total variation distance on the Heisenberg group.
The aim of Section 3 is to extend the result to the case of the higher dimensional free step 2 Carnot groups.  This is done in Theorem \ref{T2}. 
Section 4 is devoted to the gradient estimates. We first prove the $L^\infty$ gradient estimates of Corollary \ref{C-intro} and of Corollary \ref{C2}.
We then turn to the change of probability method.
 We first obtain  a  log-Harnack inequality for the semi-group in Theorem \ref{T-LH}.
Finally,  we provide the Bismut type formula in Theorem \ref{T-B},  its application in term of reverse Poincar\'e inequalities in Theorem \ref{T-RP} and  reverse weak log-Sobolev inequality in Corollary \ref{C3}. We finally deduce estimates of the gradient on the heat kernel in Corollary \ref{Prop: HeatKernel}.

\section{Description of the  Brownian motion on $\hh$}
\label{Section2}

\subsection{The Heisenberg group}\label{sec:H}

The Heisenberg group can be identified with $\R^3$ equipped with the law:
 \[
  (x_1,x_2,z)\star(x_1',x_2',z')= \left(x_1+x_1',x_2+x_2', z+z'+\frac{1}{2}(x_1 x_2'-x_2 x_1')\right).
  \]
  For our purpose, it will be convenient   to identify sometimes $\R^3$ with  $\R^2\times \R$ and to  write the law as  
\[
  (x,z)\star(x,z')= \left(x+x', z+z'+\frac{1}{2} x \cdot x' \right), 
  \]
 where  for $x=(x_1,x_2)$, $x'=(x_1,x_2)$, 
  \[
  x \cdot x' = x_1 x_2'-x_2 x_1'.
  \]

 The left invariant vector fields are given by 
 \[
 \left\{ \begin{array} {l}
 X_1 (f) (x_1,x_2,z)= \frac{d}{dt} _{|t=0} f( (x_1,x_2,z) \star (t,0,0 ) ) = \left( \partial_{x_1} - \frac{x_2}{2} \partial z \right) f(x_1,x_2,z) \\
  X_2 (f) (x_1,x_2,z)= \frac{d}{dt} _{|t=0} f( (x_1,x_2,z) \star (0,t,0 ) ) = \left( \partial_{x_2} + \frac{x_1}{2} \partial z \right) f(x_1,x_2,z)\\
Z (f) (x_1,x_2,z)= \frac{d}{dt} _{|t=0} f( (x_1,x_2,z) \star (0,0,t ) ) = \partial_z f(x_1,x_2,z).
\end{array} \right.
 \]
 
 Note that $[X_1,X_2]=Z$ and that $Z$ commutes with $X_1$ and $X_2$.
 The vectors fields $X_1,X_2$ are called the horizontal vector field whereas $Z$ is called the vertical vector field.

 \subsection{The subRiemannian Brownian motion on Heisenberg}
 
The standard (half) sub-Laplacian  on the Heisenberg group is given by $L= \frac{1}{2} (X_1^2+X_2^2)$. This is a diffusion operator and it  satisfies the H\"ormander bracket condition and thus the associated heat semigroup $P_t=e^{tL}$ admits a $\mathcal C^\infty$ positive kernel $p_t$.
 
 From a probabilistic point of view, $L$ is the generator of the following stochastic process starting in $(x_1,x_2,z)$:
 \begin{eqnarray*}
 \bbb^{(x_1,x_2,z)}
 _t&: =& (x_1,x_2,z) \star \left(  B^1_t,  B^2_t, \frac{1}{2}\left( \int_0^t B^1_s dB^2_s -  \int_0^t B^2_s dB^1_s\right) \right)\\
              &=&  \left( x_1+ B^1_t,x_2+ B^2_t, z+ \frac{1}{2} (x_1 B_t^2- x_2 B_t^1)  + \frac{1}{2}\left( \int_0^t B^1_s dB^2_s -  \int_0^t B^2_s dB^1_s\right) \right)
 \end{eqnarray*}
 where $(B^1_t,B^2_t)_{t\geq 0}$  is a standard Brownian motion on $\R^2$. 
  
It is easily seen that $(\bbb_t)_{t\geq 0} $ is a continuous process with independent and stationary increments on $\He$. We simply call it the Heisenberg Brownian motion.

The quantity
 \begin{equation} \label{eq:def-levy-area}
 A_t =\frac{1}{2} \left(\int_0^t B^1_s dB^2_s -  \int_0^t B^2_s dB^1_s \right)  
\end{equation}
is called the L\'evy area of the 2-dimensional Brownian motion.

Identifying $\mathbb{R}^3$ with $\mathbb{R}^2\times\mathbb{R}$ again, we will write $\bbb_t^{(x,z)}=(X_t,z_t)$.

\subsection{The Carnot-Carath\'eodory distance}

The sub-Laplacian $L$ is strongly related to the following  subRiemmanian distance (also called Carnot-Carath\'eodory) on $\He$:
\[
d_{\He} (a,a')= \inf_{\gamma}  \int_0 ^1  | \dot \gamma(t)  |_{\frak h} dt
\]
where $\gamma$ ranges over the horizontal curves connecting $\gamma(0)=a$ and $\gamma(1)=a'$. We remind the reader of the fact that a curve is said horizontal if it is absolutely continuous and $\dot \gamma (t) \in\Span ( X_1(\gamma(t)), X_2(\gamma(t)))$ almost surely holds. The horizontal norm $|\cdot |_{\frak h}$ is an Euclidean norm on $\Span (X_1,X_2)$ obtained by asserting that $(X_1,X_2)$ is an orthonormal basis of $\Span (X_1(a),X_2(a))$ at each point $a\in \He$. Finally the horizontal gradient $\nabla_{\frak h} f$ is $(X_1f) X_1+(X_2f)X_2$.

 The Heisenberg group admits homogeneous dilations adapted both to the distance and the group structure. They are given by 
  \[
  \dil_\lambda (x_1,x_2,z)= (\lambda x_1, \lambda  x_2, \lambda^2 z)
  \]
for $\lambda>0$. They satisfy $d_\He(\dil_\lambda(a),\dil_\lambda(a'))=\lambda d_\He(a,a')$ and, in law:
 
\[
\dil _{\frac{1}{\sqrt t}} \left( X^1_t,X^2_t, \frac{1}{2}\left( \int_0^t X^1_s dX^2_s -  \int_0^t X^2_s dX^1_s\right) \right) \overset{\mathrm{Law}}{=}   \left( X^1_1,X^2_1, \frac{1}{2}\left( \int_0^1 X^1_s dX^2_s -  \int_0^1 X^2_s dX^1_s\right) \right). 
\]

The distance is clearly left-invariant so that $\trans_p:q\in \He\mapsto p\star q$ is an isometry for every $p\in \He$. In particular
\[
d_{\He}(a,a')= d_{\He} (e, a^{-1}\star a')
\]
with $e=(0,0,0)$. Another isometry is the rotation $\rot_\theta:(x_1+ix_2,z)\in \mathbb{C}\times \R\equiv \He \mapsto (\mathrm{e}^{i\theta}(x_1+ix_2),z)$, for every $\theta\in \R$. Since  the explicit expression of $d_\He$ is not so easy, it is often simpler to work with a homogeneous quasinorm (still in the sense that the triangle inequality only holds up to a multiplicative constant). We will use 
\[
H: a=(x_1,x_2,z)\in \He\mapsto\sqrt {x_1^2+x_2^2 + |z |}\in \R ,
\]
and the attached homogeneous quasidistance $d_H(a,a')=H(a^{-1}a')$. It satisfies 
\begin{equation}\label{eq:dist-eq}
c^{-1} d_H(a,a') \leq d_{\He}(a,a') \leq c d_H(a,a')
\end{equation}
for some constant $c>1$. We finally mention $d_\He((0,0,0),(x_1,x_2,0))=\sqrt{x_1^2+x_2^2}$ and $d_\He((x_1,x_2,z),(x_1,x_2,z+h))=2\sqrt{\pi |h|}$.

\subsection{The description of the Brownian motion on $\He$ with Legendre polynomials}

Let $T>0$ 
 and  consider the scalar product defined for $f,g\in \cC([0,T],\R)$ by
\[
\langle f, g\rangle  =\int_0^T f(t)g(t) dt.
\]
Take $Q_k$ to be the associated normalized orthogonal polynomials; i.e., such that $\Vert Q_k\Vert^2= 1$.
By dilation and translation, one sees that 
\[
Q_k(x)=\sqrt{\frac 2 T } P_k\left( - 1 +\frac{2x }{T}\right)
\]
where $(P_k)_k$ are the standard (normalized) Legendre polynomials, which are orthogonal for the Lebesgue measure on $[-1,1]$.

We first consider the following representation of a  standard one-dimensional Brownian motion $(B_t)_{0\leq t\leq T}$ starting in $0$.
This   representation  is somehow close to the  standard Karhunen-Loève decomposition of the Brownian motion but as noticed in~\cite{Kuznetsov2018}, it  is well adapted to the computation of the L\'evy area. 
\begin{lemme} \label{lem:Legendre-MB}
Let 
 $(\xi_k)_{k\geq 1}$ be a sequence of  independent and identically distributed random variables  of law $\cN (0, 1)$.  Define  
\begin{equation}\label{eq:Bt-rep-pol-ortho}
B_t = \sum_{k\geq 0}  \xi_k  \int_0^t Q_k(s) ds, \; 0\leq t \leq T. 
\end{equation}
Then the process $(B_t)_{0\leq t \leq T}$ is a standard Brownian motion on $[0,T]$.
\end{lemme}

The proof is done  in \cite{Kuznetsov2018}, but let us  recall the main ideas  for the reader's convenience.
\begin{proof} 
Let $T\geq 0$ and let $(B_t)_{0\leq t \leq T}$  be defined by \eqref{eq:Bt-rep-pol-ortho}. The process $(B_t)_{0\leq t \leq T}$
is clearly a centered Gaussian process. To prove it is a standard Brownian motion, 
compute its covariance: for $0\leq s,t \leq T$
\begin{align*}
    \E[B_t B_s] &= \sum_{k \geq 0}  \left(\int_0^t Q_k(u) du \right)   \left(\int_0^s Q_k(u) du \right)\\
   &=   \sum_{k \geq 0} \langle \bone_{[0,t]},  Q_k \rangle  \,  \langle \bone_{[0,s]},  Q_k \rangle  \\
   &=  \langle \bone_{[0,t]},   \bone_{[0,s]}\rangle  =s\wedge t.
\end{align*}
where $ \langle \cdot, \cdot \rangle$ denotes the usual scalar product on $L^2([0,T])$. The result follows.
\end{proof}

We turn to the computation of the L\'evy area.

\begin{lemme}
\label{lem:Legendre-aire}
Let $(\xi_k)_{k\geq 0}$ be a sequence of independent and identically distributed random vectors   with common law $\cN (0, I_2)$.
Write $\xi_k = (\xi_k^1,\xi_k^2)^t$ and for $0\leq t \leq T$   and $i=1,2$ let
\begin{equation}\label{eq:Bt-rep-pol-ortho2}
B_t^i = \sum_{k\geq 0}  \xi_k^i  \int_0^t Q_k(s) ds. 
\end{equation}
Then $(B_t^1,B_t^2)_{0\leq t \leq T}$ is a standard two-dimensional Brownian motion and its
 associated L\'evy area $A_t:= \frac{1}{2} \Big(\int_0^t B_s^1 dB_s^2-\int_0^t B_s^2 dB_s^1 \Big) $   at the given  time $T$ may be written as
\begin{equation}\label{eq:At-rep-pol-ortho}
    A_T =  T\sum_{k\geq 0} \alpha_k \, \xi_k \cdot \xi_{k+1}
\end{equation}    
with 
\begin{equation}
    \label{eq:alpha-k}
   \alpha_k=\frac{1}{2\sqrt{4(k+1)^2-1}}=\frac{1}{2\sqrt{(2k+1)(2k+3)}} , \quad k\geq 0.
\end{equation}
\end{lemme}
As before the proof is done  in \cite{Kuznetsov2018}, but we shall  recall the main ideas  for the reader's convenience.
\begin{proof}
With the notation of Lemma \ref{lem:Legendre-aire}, 
\[
\int_0^T B_s ^1 d B_s^2= \sum_{k,l\geq 0} \xi_k^1 \xi_l^2  c_{k,l} \textrm{ with }
c_{k,l}=\int_0^T \left(\int_0^t Q_k(s) ds \right) Q_l(t) dt.\]
Now by integration by parts,
for $k,l\geq 0$,
\[
c_{k,l}= \left(\int_0^T Q_k(u) du\right) \left(\int_0^T Q_l(u) du\right)  -  c_{l,k}.
\]
Since  $Q_k$ is a family of orthogonal polynomials, one infers that for $(k,l)\neq (0,0)$, $c_{k,l}=-c_{l,k}$ and thus 
\[
c_{k,l}= 0 \textrm{ if } |k-l|\geq 2 \textrm{ or } k=l\geq 1.
\]
Therefore
\[
\int_0^T B_s ^1 d B_s^2= c_{0,0} \xi_0^1 \xi_0^2 + \sum_{k\geq 0} c_{k,k+1}(\xi_k^1 \xi_{k+1}^2 - \xi_{k+1}^1 \xi_{k}^2).
\] and thus the L\'evy area at the final time $T$ writes:
\[
A_T= \sum_{k\geq 0} c_{k,k+1}(\xi_k^1 \xi_{k+1}^2 - \xi_{k+1}^1 \xi_{k}^2).
\]The result follows by an explicit computation of the constant $c_{k,k+1}$.
\end{proof}

We recall that here in the case of the Heisenberg group: 
\begin{equation}\label{eq:Ik}
\xi_k \cdot \xi_{k+1}= \xi_k^1 \xi_{k+1}^2  -  \xi_{k+1}^1 \xi_k^2.
\end{equation}

As a direct application of Lemma \ref{lem:Legendre-aire}, the Brownian motion on $\He$ starting in $(x_1,x_2,z)$ at time $T$ may be represented by
\begin{eqnarray*}
 \bbb^{(x_1,x_2,z)}
 _T
              &=& 
              \begin{pmatrix} x_1+ \sqrt T \xi_0^1 \\
              x_2+ \sqrt T \xi_0^2  \\
              z+  \frac{
              \sqrt T}{2} (x_1 \xi_0^2- x_2 \xi_0^1)  + T \sum_{k\geq 0} \alpha_k \,  (\xi_k^1 \xi_{k+1}^2  -  \xi_{k+1}^1 \xi_k^2)
              \end{pmatrix}
 \end{eqnarray*}
 or equivalently with $x=(x_1,x_2)$,
 \begin{eqnarray*}
 \bbb_T
 &=& 
              \begin{pmatrix} x+ \sqrt T \xi_0 \\
              z+  \frac{
              \sqrt T}{2} x \cdot \xi_0 + T \sum_{k\geq 0} \alpha_k \,  \xi_k \cdot \xi_{k+1}
              \end{pmatrix}.
 \end{eqnarray*}

\subsection{Proof of Theorem \ref{T1}} 
Before we turn to the proof of Theorem \ref{T1}, we recall the standard estimate for   Gaussian vectors on $\R^d$, $d\geq 1$ (with the same identity covariance matrix).

\begin{lemme}\label{lem:couplage-gaussiennes}
Let $d\ge 1$ an integer,  $m,m'\in \R^d$,
there exists a random couple $(X,Y)$ whose marginals are  Gaussian random variables $\cN(m,I_d)$ and $\cN(m',I_d)$ and such that 
\[
\P(X \neq Y) \leq \left(\frac{\|m-m'\|_2}{\sqrt{2\pi}}\right) \wedge 1.
\]
\end{lemme}

We  provide just below  a proof of Theorem \ref{T1} with slightly weaker explicit constants $C_1$ and $C_2$. The slightly improved constants will be obtained  in Remark \ref{rem:2}. 
\begin{proof}[Proof of Theorem \ref{T1}]
For any choice of two  Heisenberg valued subRiemannian Brownian motions $((X_t,z_t))_{t\ge 0}$ and $((\tilde X_t, \tilde z_t))_{t\ge 0}$ started respectively at $(x,z)$ and $(\tilde x,\tilde z)$, we have 
\begin{equation}
    \label{eq:Aldous}
    d_{TV}\left(\mu_T^{(x,z)},\mu_T^{(\tilde x,\tilde z)}\right)\le \pr\left((X_T,z_T)\not=(\tilde X_T,\tilde z_T)\right).
\end{equation}
Consequently, to establish the estimate~\eqref{eq:main-H} it is sufficient, for each $T>0$, to find $((X_t,z_t))_{t\ge 0}$ and $((\tilde X_t, \tilde z_t))_{t\ge 0}$ started respectively at $(x,z)$ and $(\tilde x,\tilde z)$,  satisfying
\begin{equation}
\label{eq-Aldous2}
    \pr\left((X_T,z_T)\not=(\tilde X_T,\tilde z_T)\right)\le C_1\frac{\|\tilde x- x\|_2}{\sqrt T}
    +C_2\frac{|\tilde z -z -\frac12 x\cdot \tilde x|}T
\end{equation}
for $C_1,C_2>0$ not depending on $T$.

To perform the construction of the coupling,  we construct the Brownian motions $(X_t)_{t{\in[0,T]}}$ and $(\tilde X_t)_{t{\in[0,T]}}$
with Legendre polynomials as in Lemma \ref{lem:Legendre-MB}.

So let us fix $T>0$. We write 
\begin{equation}
    \label{eq:Legendre-couplage}
    \forall\ 0\le t\le T, \ X_t=x+B_t\quad \hbox{with}\quad
    B_t=\sum_{k=0}^\infty \xi_k\int_0^tQ_k(s)\,ds,
\end{equation}
where $\displaystyle \left(\xi_k=\left(\begin{array}{c}\xi_k^1\\ \xi_k^2
\end{array}\right)\right)_{k\ge 0}$ is a sequence of independent $\R^2$-valued random vectors with law $\cN(0,I_2)$.
We do the same with $(\tilde X_t)_{0\le t\le T}$, using independent $\R^2$-valued random variables $\displaystyle \left(\tilde \xi_k\right)_{k\ge 0}$  with law $\cN(0,I_2)$. Equation~\eqref{eq-Aldous2} will be obtained thanks to a well-chosen coupling of $\displaystyle \left(\xi_k\right)_{k\ge 0}$ and $\displaystyle \left(\tilde \xi_k\right)_{k\ge 0}$.

At time $T$, using Lemma \ref{lem:Legendre-aire}, we get 
\begin{equation}
    \label{eq:Legendre-couplage-X}
    X_T=x+\sqrt{T}\xi_0,\qquad  z_T= z+\frac12\sqrt{T} x\cdot \xi_0+T\sum_{k\ge 0}\alpha_k\xi_k\cdot \xi_{k+1},
\end{equation}
\begin{equation}
    \label{eq:Legendre-couplage-Xtilde}
    \tilde X_T=\tilde x+\sqrt{T}\tilde \xi_0,\qquad  \tilde z_T= \tilde z+\frac12\sqrt{T} \tilde x\cdot \tilde \xi_0+T\sum_{k\ge 0}\alpha_k\tilde \xi_k\cdot \tilde\xi_{k+1},
    \end{equation}
   where  for $k\geq 0$,  $\alpha_k$ is given by \eqref{eq:alpha-k}.

From~\eqref{eq:Legendre-couplage-X} and~\eqref{eq:Legendre-couplage-Xtilde}, we find that the coupling equation $(X_T,z_T)=(\tilde X_T,\tilde z_T)$ is equivalent to 
\begin{equation}
    \label{eq:systeme}
    \left\{
    \begin{array}{cc}
    \hfill\tilde \xi_0-\xi_0&=\frac{x-\tilde x}{\sqrt T}\hfill\\
    z-\tilde z+\frac{\sqrt T}2\left(x\cdot\xi_0-\tilde x\cdot \tilde \xi_0\right)&=T\sum\limits_{k\ge 0}\alpha_k\left(\tilde \xi_k\cdot \tilde \xi_{k+1}-\xi_k\cdot\xi_{k+1}\right).
    \end{array}
    \right.
\end{equation}
Replacing $\tilde \xi_0$ by $\xi_0+\frac{x-\tilde x}{\sqrt T}$  in the second equation we get 
\begin{equation}
    \label{eq:deuxieme-eq}
     - \zeta +(x-\tilde x)\cdot \left(\frac{\sqrt T}2 \xi_0-\sqrt{T}\alpha_0\tilde \xi_1 \right)=T\alpha_0\xi_0\cdot (\tilde \xi_1-\xi_1)
    +T\sum_{k\ge 1}\alpha_k\left(\tilde \xi_k\cdot \tilde \xi_{k+1}-\xi_k\cdot\xi_{k+1}\right).
\end{equation}
where $\zeta = \tilde z - z  - \frac{1}{2}(x \cdot \tilde x)$  is the last  coordinate in the Heisenberg group of $(x,z)^{-1} \cdot (\tilde x,\tilde z)$
We are in position to start the coupling.
We take 
\begin{equation}
    \label{eq:coupling-choice}
    \xi_k=\tilde \xi_k\quad\hbox{for all}\quad 
    k\not\in\{0,3\}.
\end{equation}
so that we are left to couple 
\begin{equation}
    \label{eq:coupling-left}
    (\xi_0,\tilde \xi_0),\quad (\xi_3,\tilde \xi_3).
\end{equation}
If~\eqref{eq:coupling-choice} is satisfied we have the simplification
\begin{align*}
   &T\alpha_0\xi_0\cdot (\tilde \xi_1-\xi_1)
    +T \sum_{k\ge 1}\alpha_k\left(\tilde \xi_k\cdot \tilde \xi_{k+1}-\xi_k\cdot\xi_{k+1}\right)\\
    &=T \alpha_{2}\left( \xi_{2}\cdot \tilde \xi_{3}  -\xi_{2}\cdot\xi_{3}\right)+ T \alpha_{3}\left(\tilde \xi_{3}\cdot  \xi_{4}-\xi_{3}\cdot\xi_{4}\right)\\
    &=T \sqrt{\alpha_{2}^2+\alpha_{3}^2} \left(\tilde \xi_{3}-\xi_{3}\right)\cdot
       \frac{\alpha_{3}\xi_{4}-\alpha_{2}\xi_{2}}{\sqrt{\alpha_{2}^2+\alpha_{3}^2}}.
\end{align*}
Define
\begin{equation}
    \label{eq:W,V}
    W=- \zeta +(x-\tilde x)\cdot \left(\frac{\sqrt T}2 \xi_0-\sqrt{T}\alpha_0 \xi_1\right) \in \R, \,
    V= \frac{\alpha_{3}\xi_{4}-\alpha_{2}\xi_{2}}{\sqrt{\alpha_{2}^2+\alpha_{3}^2}} \in \R^2.
\end{equation}
With these definitions, 
Equation~\eqref{eq:deuxieme-eq} becomes
\begin{equation}
    \label{eq:xi-V=W}
    T \sqrt{\alpha_{2}^2+\alpha_{3}^2} \left(\tilde \xi_{3}-\xi_{3}\right)\cdot V=W,
\end{equation}
 where  the random vector   $V$ is of  law $\cN(0,I_2)$ and is independent of $W$.
Consider $(E_1,E_2)$ to be a direct orthonormal basis of $\R^2$ and such that $E_1$ is proportional to $V$. Writing $U= U^1 E_1+ U^2 E_2$, and  since $E_1\cdot E_1=0 $ and $E_1\cdot E_2=1$,  the solutions of  equation 
\begin{equation}\label{eq:UV=W}
U \cdot V =W
\end{equation} are precisely the vectors $U \in \R^2$ such that
\[
U^2= - \frac{W}{\Vert V \Vert_2}.
\]
Note that nothing is imposed on the coordinate $U^1$. 
A solution of \eqref{eq:deuxieme-eq} is thus obtained if 
\begin{equation}\label{eq-sol-coupling-tilde}
\tilde \xi_{3}-\xi_{3} = - \frac{1}{ T \sqrt{\alpha_{2}^2+\alpha_{3}^2}} \frac{W}{\Vert V \Vert_2} E_2.
\end{equation}

We also denote $(F_1,F_2)$ to be the direct orthonormal basis of $\R^2$ and such that $F_1$ is proportional to $\tilde x-x$. We emphasize that $W$ depends only on $\langle \xi_0, F_2 \rangle $  (and on $\langle \xi_1, F_2 \rangle $)
and thus we will also take
$\langle \tilde \xi_0, F_2 \rangle =\langle \xi_0, F_2 \rangle$.

Now by Lemma \ref{lem:couplage-gaussiennes},  given  the values of  $\xi_k$ for  $k\in \N \setminus\{0,3\}$ and the value of $\langle \xi_0, F_2 \rangle $, it is possible to construct a coupling of the three dimensional Gaussian random vectors $(\langle \xi_0, F_1 \rangle,\xi_3)$ and $( \langle \tilde \xi_0, F_2 \rangle, \tilde \xi_3)$  such that 
\begin{align*}
    \pr\left((X_T,z_T)\not=(\tilde X_T,\tilde z_T) | (\xi_k)_{ k\in \N \setminus\{0,3\}},  \langle \xi_0, F_2 \rangle \right)
    \leq \frac{1}{\sqrt{2\pi}} \left( \frac{\Vert \tilde x- x\Vert_2}{\sqrt T} +  \frac{1}{ T \sqrt{\alpha_{2}^2+\alpha_{3}^2}} \frac{|W|}{\Vert V \Vert_2} \right)
\end{align*}
and thus since $V$ and $W$ are independent
\[
\pr\left((X_T,z_T)\not=(\tilde X_T,\tilde z_T)\right) \leq \frac{1}{\sqrt{2\pi}} \left( \frac{\Vert x-\tilde x\Vert_2 }{\sqrt T} +  \frac{1}{ T \sqrt{\alpha_{2}^2+\alpha_{3}^2}} \E\left[\frac{1}{\Vert V \Vert}_2 \right] \E[|W|] \right).
\]
Now since $V$ is a random vector with law $\cN(0,I_2)$:
\[
\E\left[\frac{1}{\Vert V \Vert_2} \right] = \sqrt {\frac{\pi}{2}}.
\]
Denoting  $\hat \xi_0 =  \frac{1}{\sqrt{\frac {1}{4} + \alpha_0^2}}  \left( \frac{\sqrt T}2 \xi_0-\sqrt{T}\alpha_0 \xi_1\right)  \sim \cN(0,I_2)$
and with the same orthonormal basis $(F_1,F_2)$ of $\R^2$:
\begin{align*}
\E[ |W| ] & \leq   |\zeta| +   \sqrt T \sqrt{\frac {1}{4} + \alpha_0^2}  \;  \E \left[ \left|(\tilde x-x )\cdot \hat \xi_0 \right|\right]\\
      & = |\zeta| +   \sqrt T \sqrt{\frac {1}{4} + \alpha_0^2}  \;  \Vert \tilde x -x \Vert_2 \,   \E \big[   |\langle F_2 , \hat \xi_0 \rangle  |\big]\\ 
         & =   |\zeta| +    \sqrt{\frac {T}{3}} \,   \sqrt {\frac{2}{\pi}}   \,  \Vert \tilde x -x \Vert
_2\end{align*} 
since $\langle F_2 , \hat \xi_0 \rangle \sim \cN(0,1)$.
Thus the conclusion \eqref{eq:main-H} holds
with 
\[
C_2= \frac{1}{\sqrt {2\pi}}  \frac{  1} {\sqrt{\alpha_{2}^2+\alpha_{3}^2}}  \sqrt {\frac{\pi}{2}} 
= \sqrt {22.5}
\textrm{ and }
C_1= \frac{1}{\sqrt {2\pi}} + \sqrt {\frac{2}{3 \pi}} C_2 
= \frac{1}{\sqrt {2\pi}} \left(1 + \sqrt{30}\right).
\]
The constants given in Theorem \ref{T1} will be obtained in Remark \ref{rem:2}.
\end{proof}

\begin{NB}\label{rem:1}
The above explicit constant $C_1$ and $C_2$ are not optimal. In some sense, we try to use the less noise possible in the coupling.   It should be possible to decrease their values by allowing more random Gaussian vectors to be different in \eqref{eq:coupling-choice}.
\end{NB}

\begin{NB}\label{rem:2}
Using the left invariance and the rotational invariance of the Heisenberg group, it is enough to consider the total variation between the measures 
$\mu_T^{(x_1,0,z)}$ and $\mu_T^{(0,0,0)}$.
In this case, 
 we can take
 \[
 \tilde \xi_0^2= \xi_0^2, \tilde \xi_2^1=\xi_2^1, \tilde \xi_3^1=\xi_3^1 \textrm{ and }  
 \tilde \xi_k = \xi_k  \textrm{ for  } k =1 \textrm{ and } k\geq 4.
 \]
so that we are left to couple 
\[
    (\xi_0^1,\xi_2^2,\xi_3^2),\quad  ( \tilde \xi_0^1, \tilde \xi_2^2, \tilde \xi_3^2).
\]
By rewriting carefully the above proof, one can then replace the constant 
\[
\frac{1}{  \sqrt{\alpha_{2}^2+\alpha_{3}^2}} \E\left[\frac{1}{\Vert V \Vert} \right] \textrm{ by }
                               \E \left[\frac{1}{\sqrt{(\alpha_1^2 + \alpha_2^2) Z_1^2 + (\alpha_2^2+ \alpha_3^2) Z_2^2}}\right]
\]
where $Z_1$ and $Z_2$ are two independent $\cN(0,1)$ random variables. In fact, in view of Remark \ref{rem:1}, if one allows to couple,
\[
    \left(\xi_0^1,(\xi_2^2,\xi_3^2), (\xi_5^2,\xi_6^2), (\xi_8^2,\xi_9^2), \dots \right),\quad  \left( \tilde \xi_0^1, (\tilde \xi_2^2, \tilde \xi_3^2), (\tilde \xi_5^2,\tilde \xi_6^2), (\tilde \xi_8^2,\tilde \xi_9^2), \dots\right),
\]
the previous constant may even  be replaced by 
\[
 \E \left[\frac{1}{\sqrt{ \sum_{k \geq 1} c_k^2  Z_k^2 } }\right]
\]
where $(Z_n)_{n\geq 1}$ is an independent sequence of $\cN(0,1)$ random variables and where $c^2$ is the sequence:
\begin{align*}
c^2&=\left( \alpha_1^2 + \alpha_2^2, \alpha_2^2+ \alpha_3^2,  \alpha_4^2 + \alpha_5^2,  \alpha_5^2 + \alpha_6^2  ,  \alpha_7^2 + \alpha_8^2, \alpha_8^2 + \alpha_9^2, \dots  \right)\\
&=\left( \frac{1}{2\times 3\times 7},\frac{1}{2\times 5\times 9}, \frac{1}{2\times 9\times 13},  \frac{1}{2\times 11\times 15},  \frac{1}{2\times 15\times 19}, \frac{1}{2\times 17\times 21}, \dots  \right)\\
&\geq \gamma(1, \frac  1 4, \frac 1 9, \frac 1{16},\frac 1{25},\frac 1{36},\dots) 
\end{align*}
with \[
\gamma= \frac{1}{42}
\]
and since one has for $k\geq 0$
\begin{equation}\label{E-kk+1}
\alpha_k^2 + \alpha_{k+1}^2= \frac{1}{2(2k+1)(2k+5)}.
\end{equation}

This gives
\[
 \E \left[\frac{1}{\sqrt{ \sum_{k \geq 1} c_k^2  Z_k^2 } }\right] \leq \frac{\sqrt 2} {\pi \sqrt \gamma}
 \E \left[ S_{\frac 1 2}  ^{-1/2} \right] \leq \frac{5} {\pi \sqrt \gamma}= \frac{5 \sqrt {42}} {\pi }
\]
where   $\frac{2}{\pi^2}\sum_{k \geq 1} \frac{1}{k^2}  Z_k^2 $ has the same law as $S_{\frac 1 2}$  defined in Lemma \ref{L5} and using \eqref{E58.3bis} in this lemma.

The announced slitghly better constants ensue:
\[
C_2=\frac{5 \sqrt {21}}{\pi \sqrt{\pi}} 
\textrm{ and }
C_1=\frac{1}{\sqrt{2\pi}} + \sqrt{\frac{2}{3\pi}} C_2=\frac{1}{\sqrt{2\pi}} \left(1+ \frac{5\sqrt{28}}{\pi \sqrt \pi}\right).
\]
\end{NB}

\section{Distance in total variation of subelliptic Brownian motions in  Carnot groups}
\label{Section3}
The aim of this section is to extend Theorem~\ref{T1} to free step $2$ Carnot groups. We start with several definitions and lemmas.

\subsection{Some preliminaries}

For $n,m$ positive integers, denote by $M_{n,m}(\R)$ the set of matrices with real entries,~$n$ lines and $m$ columns. Also denote by $\frak{so}(n)$ the set of skew-symmetric matrices of size $n$. If $u,v\in M_{n,1}(\R)$, let $u\symp v:= uv^t-vu^t\in \frak{so}(n)$, where $v^t\in M_{1,n}(\R)$ denotes the transpose of $v$. 

In the sequel we will identify $M_{n,1}(\R)$ with $\R^n$.
\begin{NB}
\label{NB1}
In the special case $n=3$, $\frak{so}(3)$ is usually identified with $\R^3$ via the linear map
\begin{equation}
    \begin{split}
        \psi : \R^3&\to \frak{so}(3)\\
        \left(\begin{array}{c}a\\b\\c\end{array}\right)&\mapsto\left(\begin{array}{ccc}0&-c&b\\c&0&-a\\-b&a&0\end{array}\right).
    \end{split}
\end{equation}
On the other hand, $\R^3$ is endowed with the usual vectorial product $\wedge$. In this situation, it can be checked that for $u,v\in \R^3$
    \begin{equation}
        \label{E4}
        \psi(u\wedge v)=[\psi(u),\psi(v)]=-u\symp v.
    \end{equation}
    In other words, the map $\psi$ is an isomorphism between the two Lie algebras $(\R^3,\wedge)$ and $(\frak{so}(3),[\cdot,\cdot])$.
\end{NB}
\begin{defi}
For $n\ge 2$, the free step $2$ Carnot group $\Ge_n$ is the vector space
\begin{equation}
\label{E2}
\Ge_n:=\R^n\times \frak{so}(n)
\end{equation}
endowed with the group operation
\begin{equation}
    \label{E3}
    (u,A)\star (v,B):=\left(u+v, A+B+\frac12 u\symp v\right).
\end{equation}
\end{defi}
\begin{NB}
\label{NB2}
The Carnot group $\Ge_2$ is isomorphic to the Heisenberg group $\He$.
\end{NB}

Consider an integrable random variable $W$ taking its values in $\frak{so}(n)$, and for $m\ge n+2$, $V_1, \ldots ,V_{m}$, $m$ independent random vectors taking their values in $\R^n$, with law $\cN(0,I_n)$ and independent of $W$. Our next aim is to solve in $U_1,\ldots,U_{m}$ random variables with values in $\R^n$, the equation
\begin{equation}
    \label{E12}
    \sum_{k=1}^{m}U_k\symp V_k=W.
\end{equation}
Clearly the solution is not unique. We will make a specific  choice which will together give uniqueness and allow explicit computations.
Letting $(e_1,\ldots, e_n)$ be the canonical basis of $\R^n=M_{n,1}(\R)$, we have the canonical decomposition

\begin{equation}
    \label{E11}
    W=\sum_{i,j=1}^nW^{i,j}e_i e_j^t=\frac12\sum_{i,j=1}^nW^{i,j}e_i\symp e_j,\quad\hbox{since
    }\quad W^{j,i}=-W^{i,j}.
\end{equation}
We will denote
\begin{equation}
    \label{E13.1}
    U_k=\sum_{j=1}^n U_k^j e_j,\qquad k=1,\ldots, m,
\end{equation}
\begin{equation}
    \label{E13}
    V_k=\sum_{j=1}^n V_k^j e_j,\qquad k=1,\ldots, m,
\end{equation}
\begin{equation}
    \label{E14}
    \SU=\left(\begin{array}{cccc}
    U_1^1&U_2^1&\cdots&U_{m}^1\\
    U_1^2&U_2^2&\cdots&U_{m}^2\\
    \vdots&\vdots&\vdots&\vdots\\
    U_1^n&U_2^n&\cdots&U_{m}^n
    \end{array}
    \right)=\left(U_k^j\right)_{1\le j\le n, \ 1\le k\le m},
    \end{equation}
    \begin{equation}\label{E14.1}
    \SV=\left(\begin{array}{cccc}
    V_1^1&V_2^1&\cdots&V_{m}^1\\
    V_1^2&V_2^2&\cdots&V_{m}^2\\
    \vdots&\vdots&\vdots&\vdots\\
    V_1^n&V_2^n&\cdots&V_{m}^n
    \end{array}
    \right)=\left(V_k^i\right)_{1\le i\le n,\  1\le k\le m}
\end{equation}
and 
\begin{equation}
    \label{E15}
    \SW=\left(\begin{array}{ccccc}
    0&W^{1,2}&W^{1,3}&\cdots&W^{1,n}\\
    -W^{1,2}&0&\ddots&\cdots&W^{2,n}\\
    \vdots&\ddots&\ddots&\ddots&\vdots\\
    \vdots&\vdots&\ddots&\ddots&W^{n-1,n}\\
    -W^{1,n}&\cdots&\cdots&-W^{n-1,n}&0
    \end{array}
    \right)=\left(W^{i,j}\right)_{1\le i\le n,\  1\le j\le n}.
\end{equation}
Equation~\eqref{E12} rewrites as
\begin{equation}
    \label{E16-equivalence}
   \SU \SV^t - \SV \SU^t=\SW.
\end{equation}

Equation \eqref{E16-equivalence} is  sometimes called a $T$-Sylvester  equation in the literature. 
The next proposition provides a  particular  solution and gives some estimates when $m\geq n+2$.

\begin{prop}
\label{P1}
A solution to Equation~\eqref{E12}  is given by 
\begin{equation}
    \label{E17}
   \SU^t=-\frac12\SV^t\left(\SV\SV^t\right)^{-1}\SW.
\end{equation}
For $0<q\leq 2$, it satisfies:
\begin{equation}
    \label{E20}
  \E[\|\SU\|^q]
 \leq  \frac1{ 2^q n^\frac{q}{2}}\E\left[  \tr\left((\SV\SV^t)^{-1}\right)^\frac{q}{2}\right]  \E\left[\|\SW\|^q\right],
\end{equation}
where $\|\SU\|$ and $\|\SW\|$ denote Hilbert-Schmidt norms.

In particular, it satisfies 
\begin{equation}
    \label{E20-ter}
  \E[\|\SU\|^2]\le \frac{1}{4(m-n-1)}\E\left[\|\SW\|^2\right].  
\end{equation}
\end{prop}
\begin{proof}
We first note that we have a particular  solution of Equation \eqref{E16-equivalence} if
\begin{equation}
    \label{E16}
    \SV \SU^t=-\frac12\SW.
\end{equation}
We easily check that $\SU$ given by~\eqref{E17} is a solution of ~\eqref{E16} and thus also of \eqref{E12}.

From~\eqref{E17} we get 
\begin{equation}
    \label{E20.2}
    \SU \SU^t=\frac14\SW^t\left(\SV \SV^t\right)^{-1}\SW.
\end{equation}
This yields
\begin{equation}
    \label{E20.3}
    \|\SU\|=\sqrt{\tr\left(\SU \SU^t\right)}=\frac12\sqrt{\tr\left(\SW^t\left(\SV \SV^t\right)^{-1}\SW\right)}.
 \end{equation} 
  On the other hand, $\SV$ is independent of $\SW$ and $\SV\SV^t$  is a standard Wishart matrix $\W(n,m)$ of size $n\times n$ and with  $m$ degree of freedom and thus  can be written in singular value decomposition as  
  \begin{equation}
      \label{E20.6}
      \SV\SV^t=\SSS^t\SD^2\SSS
  \end{equation}
  where $\SSS$ and $\SD$  are two independent random variables taking their values respectively in $O(n)$ and $M_{n,n}(\R)$, $\SSS$ having uniform law and $\SD$ being diagonal with positive eigenvalues $0<d_1<\ldots <d_n$. From this  and with the conditional Jensen inequality since $q\leq 2$, we get: 
  \begin{eqnarray*}
     2^q \E\left[\Vert \SU \Vert^q\right]&=&
     \E\left[ \tr\left(\SW^t\left(\SV \SV^t\right)^{-1}\SW\right)^\frac{q}{2}\right]
     =\E\left[\tr\left(\SW^t\SSS^t\SD^{-2}\SSS\SW\right)^\frac{q}{2}\right]\\  &=&
     \E\left[\tr\left(\SSS\SW\SW^t\SSS^t\SD^{-2}\right)^\frac{q}{2}\right]
     =
     \E\left[ \left(\sum_{i=1}^n d_i^{-2} e_i^t\SSS\SW\SW^t\SSS^t e_i\right)^\frac{q}{2}\right]\\
     &=&\E\left[ \E\left[\left(\sum_{i=1}^n d_i^{-2} e_i^t\SSS\SW\SW^t\SSS^t e_i \right)^\frac{q}{2}| \SW,\SD\right]\right]\\
     &\le & \E\left[ \left(\E\left[\sum_{i=1}^n d_i^{-2} e_i^t\SSS\SW\SW^t\SSS^t e_i| \SW,\SD\right]\right)^\frac{q}{2}\right].
     \end{eqnarray*}
     Now for all $1\leq i \leq n$,
  $$
  \E\left[  e_i^t\SSS\SW\SW^t\SSS^t e_i|\SW,\SD\right]=\frac1n\tr(\SW\SW^t)
  $$
  since $\SSS$ is uniformly distributed and independent of $\SW$ and $\SD$. We get from this 
  \[
   2^q \E\left[\Vert \SU \Vert^q\right] \leq  \frac1{ n^\frac{q}{2}}\E\left[      \tr(\SW\SW^t)^\frac{q}{2} \tr\left(\SD^{-2}\right)^\frac{q}{2} \right]
   \]
   
 and Inequality \eqref{E20} follows since $\tr\left(\SD^{-2}\right)=\tr\left(\left(\SV\SV^t\right)^{-1}\right)$ is independent of $\SW$. 
  
      By~\cite{PH21} \emph{Example~3.1} we have
      \begin{equation}
          \label{E20.8}
   \E\left[
      \tr\left(\left(\SV\SV^t\right)^{-1}\right)
      \right]=\frac{n}{m-n-1}
      \end{equation}
      and Inequality \eqref{E20-ter} directly follows.
\end{proof}

\subsection{Distance in total variation of two Brownian motions in Carnot groups}

 Theorem~\ref{T1} yields an upper bound for the total variation distances between the laws of two subRiemannian Brownian motions in $\He=\Ge_2$  at time~$T$ started at different points. The aim of this section is to extend the result to $\Ge_n$-valued subRiemannian Brownian motions for all $n\geq 3$.

\begin{defi}
\label{D1}
A subRiemannian Brownian motion in $\Ge_n$ started at $(x,z)$ is a process $((X_t,z_t))_{t\ge 0}$ such that $(X_t)_{t\ge 0}$ is a $\R^n$-valued Brownian motion started at $x$ and $(z_t)_{t\ge 0}$ is the $\frak{so}(n)$-valued process satisfying
\begin{equation}
    \label{E21}
   \forall\  t\ge 0,\quad z_t=z+\frac12\int_0^t X_s\symp dX_s.
\end{equation}
\end{defi}

\begin{thm}
\label{T2}
For $T>0$ and $(x,z)\in \Ge_n$ let $\mu_T^{(x,z)}$ be the law at $T$ of the subRiemannian Brownian motion started at $(x,z)$.
We have for all $T>0$ and all $((x,z),(\tilde x,\tilde z))\in \Ge_n^2$, 
\begin{equation}
    \label{E27}
    d_{TV}\left(\mu_T^{(x,z)},\mu_T^{(\tilde x,\tilde z)}\right)\le C_1(n)\frac{\|\tilde x- x\|_2}{\sqrt T}
    +C_2(n)\frac{\|\tilde z -z -\frac12 x\symp \tilde x\|}T
\end{equation}
where 
\begin{equation}
    \label{E27.1}
    C_2(n):=\frac{1}{\sqrt{\pi}}\left(6\sqrt{n}+\frac4{\sqrt{n}}\right)\quad\hbox{and}\quad C_1(n):=\frac1{\sqrt{2\pi}}+\sqrt{\frac{2(n-1)}3}C_2(n).
\end{equation}
\end{thm}

\begin{NB}
Theorem \ref{T2} also applies when $n=2$, i.e., in the case of the Heisenberg group.
In order to compare the constants in Theorems \ref{T1} and \ref{T2}, note that $\|\tilde z -z -\frac12 x\symp \tilde x\|= \sqrt 2 |\tilde z-z -\frac{1}{2} (x_1\tilde x_2- x_2 \tilde x_1 )|$.
\end{NB}
\begin{proof}
For any choice of two  $\Ge_n$-valued subRiemannian Brownian motions $((X_t,z_t))_{t\ge 0}$ and $((\tilde X_t, \tilde z_t))_{t\ge 0}$ started respectively at $(x,z)$ and $(\tilde x,\tilde z)$, we have 
\begin{equation}
    \label{E28}
    d_{TV}\left(\mu_T^{(x,z)},\mu_T^{(\tilde x,\tilde z)}\right)\le \pr\left((X_T,z_T)\not=(\tilde X_T,\tilde z_T)\right).
\end{equation}
Consequently, to establish the estimate~\eqref{E27} it is sufficient, for each $T>0$, to find $((X_t,z_t))_{t\ge 0}$ and $((\tilde X_t, \tilde z_t))_{t\ge 0}$ started respectively at $(x,z)$ and $(\tilde x,\tilde z)$,  satisfying
\begin{equation}
\label{E30}
    \pr\left((X_T,z_T)\not=(\tilde X_T,\tilde z_T)\right)\le C_1(n)\frac{\|\tilde x- x\|_2}{\sqrt T}
    +C_2(n)\frac{\|\tilde z -z -\frac12 x\symp \tilde x\|}T.
\end{equation}

Adopting the same strategy as in Section~\ref{Section2}, we construct the Brownian motions  $(X_t)_{t\ge 0}$ and $(\tilde X_t)_{t\ge 0}$ with Legendre polynomials. 

Fix $T>0$. Similarly to Equation~\eqref{eq:Bt-rep-pol-ortho2} but now in dimension $n$, we write 
\begin{equation}
    \label{E22}
    \forall\ 0\le t\le T, \ X_t=x+B_t\quad \hbox{with}\quad
    B_t=\sum_{k=0}^\infty \xi_k\int_0^tQ_k(s)\,ds,
\end{equation}
where $\displaystyle \left(\xi_k=\left(\begin{array}{c}\xi_k^1\\\vdots\\\xi_k^n
\end{array}\right)\right)_{k\ge 0}$ is a sequence of independent $\R^n$-valued random vectors with law $\cN(0,I_n)$.
We do the same with $(\tilde X_t)_{0\le t\le T}$, using independent $\R^n$-valued random variables $\displaystyle \left(\tilde \xi_k\right)_{k\ge 0}$  with law $\cN(0,I_n)$. Equation~\eqref{E30} will be obtained thanks to a well-chosen coupling of $\displaystyle \left(\xi_k\right)_{k\ge 0}$ and $\displaystyle \left(\tilde \xi_k\right)_{k\ge 0}$.

At time $T$ we get 
\begin{equation}
    \label{E23}
    X_T=x+\sqrt{T}\xi_0,\qquad  z_T= z+\frac12\sqrt{T} x\symp \xi_0+T\sum_{k\ge 0}\alpha_k\xi_k\symp \xi_{k+1},
\end{equation}
\begin{equation}
    \label{E24}
    \tilde X_T=\tilde x+\sqrt{T}\tilde \xi_0,\qquad  \tilde z_T= \tilde z+\frac12\sqrt{T} \tilde x\symp \tilde \xi_0+T\sum_{k\ge 0}\alpha_k\tilde \xi_k\symp \tilde\xi_{k+1},
    \end{equation}
   where $(\alpha_k)_{k\ge 0}$ is defined in \eqref{eq:alpha-k}.

From~\eqref{E23} and~\eqref{E24}, we find that the coupling equation $(X_T,z_T)=(\tilde X_T,\tilde z_T)$ is equivalent to 
\begin{equation}
    \label{E25}
    \left\{
    \begin{array}{cc}
    \hfill\tilde \xi_0-\xi_0&=\frac{x-\tilde x}{\sqrt T}\hfill\\
    z-\tilde z+\frac{\sqrt T}2\left(x\symp\xi_0-\tilde x\symp \tilde \xi_0\right)&=T\sum\limits_{k\ge 0}\alpha_k\left(\tilde \xi_k\symp \tilde \xi_{k+1}-\xi_k\symp\xi_{k+1}\right).
    \end{array}
    \right.
\end{equation}
Replacing $\tilde \xi_0$ by $\xi_0+\frac{x-\tilde x}{\sqrt T}$  in the second equation we get
\begin{equation}
    \label{E26}
    - \zeta+(x-\tilde x)\symp \left(\frac{\sqrt T}2 \xi_0-\sqrt{T}\alpha_0\tilde \xi_1\right)=T\alpha_0\xi_0\symp (\tilde \xi_1-\xi_1)
    +T\sum_{k\ge 1}\alpha_k\left(\tilde \xi_k\symp \tilde \xi_{k+1}-\xi_k\symp\xi_{k+1}\right)
\end{equation}
where $\zeta= \tilde z - z - \frac 1 2 x\symp \tilde x$. 
We are in position to start the coupling. As in the previous section we let $m\ge n+2$, that we will choose at the end.
We take 
\begin{equation}
    \label{E32}
    \xi_k=\tilde \xi_k\quad\hbox{for all}\quad 
    k\not\in\{0,3,6,\ldots 3m\},
\end{equation}
so that we are left to couple 
\begin{equation}
    \label{E33}
    (\xi_0,\tilde \xi_0),\quad (\xi_3,\tilde \xi_3),\ldots
    (\xi_{3m},\tilde \xi_{3m}).
\end{equation}
If~\eqref{E32} is satisfied we have the simplification
\begin{align*}
   &T\alpha_0\xi_0\symp (\tilde \xi_1-\xi_1)
    +T\sum_{k\ge 1}\alpha_k\left(\tilde \xi_k\symp \tilde \xi_{k+1}-\xi_k\symp\xi_{k+1}\right)\\
    &=T\sum_{k=1}^{m}\left(\alpha_{3k-1}\left( \xi_{3k-1}\symp \tilde \xi_{3k}-\xi_{3k-1}\symp\xi_{3k}\right)+\alpha_{3k}\left(\tilde \xi_{3k}\symp  \xi_{3k+1}-\xi_{3k}\symp\xi_{3k+1}\right)\right)\\
    &=\sum_{k=1}^{m}\left(\tilde \xi_{3k}-\xi_{3k}\right)\symp
     T\left(\alpha_{3k}\xi_{3k+1}-\alpha_{3k-1}\xi_{3k-1}\right)\\
     &=\sum_{k=1}^{m}T\sqrt{\alpha_{3k}^2+\alpha_{3k-1}^2}\left(\tilde \xi_{3k}-\xi_{3k}\right)\symp
     \frac{\alpha_{3k}\xi_{3k+1}-\alpha_{3k-1}\xi_{3k-1}}{\sqrt{\alpha_{3k}^2+\alpha_{3k-1}^2}}.
\end{align*}
Define
\begin{equation}
    \label{E31}
    W=-\zeta +(x-\tilde x)\symp \left(\frac{\sqrt T}2 \xi_0-\sqrt{T}\alpha_0\tilde \xi_1\right) ,
\end{equation}
\begin{equation}
    \label{E34}
    V_k=\frac{\alpha_{3k}\xi_{3k+1}-\alpha_{3k-1}\xi_{3k-1}}{\sqrt{\alpha_{3k}^2+\alpha_{3k-1}^2}},\quad k=1,\ldots, m.
\end{equation}
With these definitions, 
Equation~\eqref{E26} becomes
\begin{equation}
    \label{E37}
    \sum_{k=1}^{m} T\sqrt{\alpha_{3k}^2+\alpha_{3k-1}^2}\left(\tilde \xi_{3k}-\xi_{3k}\right)\symp V_k=W,
\end{equation}
and the random vectors $V_k$, $1\le k\le m$ are independent with the same law $\cN(0,I_n)$.

Let $(U_1,\ldots, U_{m})$ be the solution given by~\eqref{E17} to Equation~\eqref{E12}  $\displaystyle \sum_{k=1}^{m} U_k\symp V_k=W$. Using~\eqref{E37} we see that a solution to~\eqref{E26} is given by 
\begin{equation}
    \label{E35}
    \tilde \xi_{3k}-\xi_{3k}=\frac{U_k}{T\sqrt{\alpha_{3k}^2+\alpha_{3k-1}^2}}=:\hat U_k,\quad k=1,\ldots, m.
\end{equation}
Define
\begin{equation}
    \label{E35.1}
    \xi=\left(\begin{array}{c}\xi_1\\\xi_2\\\vdots\\\xi_m\end{array}\right),\quad \tilde\xi=\left(\begin{array}{c}\tilde\xi_1\\\tilde\xi_2\\\vdots\\\tilde\xi_m\end{array}\right),\quad\hat U=\left(\begin{array}{c}\hat U_1\\\hat U_2\\\vdots\\\hat U_m\end{array}\right)
\end{equation}
with the $\hat U_k$ defined in~\eqref{E35}. The random vectors $\xi$, $\tilde \xi$ and $\hat U$ take their values in $M_{nm,1}(\R)$ and $\xi$, $\tilde \xi$ have law $\cN(0,I_{nm})$.
Recalling the system~\eqref{E25}, we obtain with~\eqref{E35} that 
\begin{equation}
    \label{E38}
    \begin{split}
    \pr\left((X_T,z_T)\not=(\tilde X_T,\tilde z_T)\right)
    \le
    \pr\left(\tilde \xi_0-\xi_0\not=\frac{x-\tilde x}{\sqrt T}\right)+\pr\left(
    \tilde \xi-\xi\not=\hat U\right).
    \end{split}
\end{equation}
Observing that  the random vector $\hat U$ is independent of $\xi$, and using Lemma~\ref{lem:couplage-gaussiennes}, we get the estimate
\begin{equation}
    \label{E40}
    \pr\left((X_T,z_T)\not=(\tilde X_T,\tilde z_T)\right)\le
    \frac{\|x-\tilde x\|_2}{\sqrt {2\pi T}}+\frac{\E[\|\hat U\|]}{\sqrt{2\pi}}.
\end{equation}
By~\eqref{eq:alpha-k} the sequence $(\alpha_k)_{k\ge 0}$ is decreasing, consequently the sequence  $\displaystyle \left(\frac{1}{\sqrt{\alpha_{3k}^2+\alpha_{3k-1}^2}}\right)_{k\ge 0}$ is increasing and $\displaystyle \E[\|\hat U\|]\le \frac{\E[\| U\|]}{ T \sqrt{\alpha_{3m}^2+\alpha_{3m-1}^2}}$. On the other hand 
using~\eqref{eq:alpha-k} or \eqref{E-kk+1},
we have for $k\ge 1$,
\begin{equation}
    \label{E42}
    \frac{1}{\alpha_{3k}^2+\alpha_{3k-1}^2}=
   2 (6k-1) (6k+3) \le8(3k+1)^2.
\end{equation}

Recalling that by Proposition~\ref{P1}, and working for simplicity with  $q=2$,
$$
\E[\|U\|^2]=\E[\|\SU\|^2]\le \frac{1}{4(m-n-1)}\E\left[\|W\|^2\right]
$$
we get 
\begin{equation}
    \label{E42.1}
    \E[\|\hat U\|^2]\le \frac{ 2(3m+1)^2}{ T^2(m-n-1)}\E\left[\|W\|^2\right]
\end{equation}
On the other hand, writing from~\eqref{E31}
$$
 W=- \zeta +(x-\tilde x)\symp \left(\frac{\sqrt T}2 \xi_0-\sqrt{T}\alpha_0\tilde \xi_1\right),
$$
we get 
$$
\E[\left\|W\|^2\right]=\left\|\zeta \right\|^2+\E\left[\left\|(x-\tilde x)\symp \left(\frac{\sqrt T}2 \xi_0-\sqrt{T}\alpha_0\tilde \xi_1\right)\right\|^2\right].
$$
We will do the computation in an orthonormal basis $(E_1,\ldots, E_n)$ of $\R^n$ such that $x-\tilde x=\|x-\tilde x\|_2 E_1$. Since $\alpha_0=\frac1{2\sqrt{3}}$ we have $\frac{\sqrt T}2 \xi_0-\sqrt{T}\alpha_0\tilde \xi_1= \sqrt{\frac{T }{3} } \hat\xi_0$ where $\hat \xi_0$ is a $\R^n$-valued Gaussian random variable with law $\cN(0,I_n)$. Writing $\hat \xi_0=\sum\limits_{i=1}^n\hat\xi_0^i E_i$ we obtain 
$$
(x-\tilde x)\symp \left(\frac{\sqrt T}2 \xi_0-\sqrt{T}\alpha_0\tilde \xi_1\right)= \sqrt{ \frac T 3} \|x-\tilde x\|_2 \, \sum_{i=2}^n\hat \xi_0^i E_1\symp E_i.
$$
The matrices $E_1\symp E_i=E_1E_i^t-E_iE_1^t$ being orthogonal each with norm $\sqrt2$ we obtain
$$
\E\left[\left\|(x-\tilde x)\symp \left(\frac{\sqrt T}2 \xi_0-\sqrt{T}\alpha_0\tilde \xi_1\right)\right\|^2\right]=\|x-\tilde x\|_2^2\,  \frac{2T(n-1)}3.
$$
We get
 \begin{equation}\label{E42.3}
     \E\left[\|W\|^2\right]=
     \left\|\zeta \right\|^2+{\frac{2T(n-1)}3}\|x-\tilde x\|_2^2.
 \end{equation}
 Using this estimate in~\eqref{E42.1} yields
\begin{equation}
    \label{E42.4}
    \E[\|\hat U\|^2]\le \frac{2(3m+1)^2}{T^2 (m-n-1)}\left(\left\|\zeta \right\|^2+\frac{2T(n-1)}3\|x-\tilde x\|_2^2\right).
\end{equation}
We can easily prove that the best choice for an integer $m$ is 
\begin{equation}
    \label{E42.5}
    m=2n+1\quad \hbox{implying}\quad \frac{2(3m+1)^2}{(m-n-1)}=\left(6\sqrt{2n}+\frac{4\sqrt2}{\sqrt{n}}\right)^2.
\end{equation}
So together with~\eqref{E40},
\begin{equation}
    \label{E42.41}
    \begin{split}
     &\pr\left((X_T,z_T)\not=(\tilde X_T,\tilde z_T)\right)\le\\&
    \frac{\|x-\tilde x\|_2}{\sqrt {2\pi T}}+
    \frac{1}{T\sqrt{\pi}}\left(6\sqrt{n}+\frac4{\sqrt{n}}\right)
     \left(\left\|\zeta \right\|+\sqrt{\frac{2T(n-1)}3}\|x-\tilde x\|_2\right).
     \end{split}
\end{equation}
We obtain the wanted inequality~\eqref{E27} with 
\begin{equation}
    \label{E41}
     C_2(n)=\frac{1}{\sqrt{\pi}}\left(6\sqrt{n}+\frac4{\sqrt{n}}\right)\quad \hbox{and}\quad
    C_1(n)=\frac{1}{\sqrt {2\pi}}+\sqrt{\frac{2(n-1)}3}C_2(n).
\end{equation}
\end{proof}
\section{Application to gradients inequalities}

\subsection{Direct estimates for the horizontal and vertical gradient}

Similarly to the case of the Heisenberg group, it is possible to define the left-invariant vector fields on $\Ge_n$.
The horizontal vector fields are defined for $1\leq i \leq n $ by 
\[
 X_i(f) (x,z)= \frac{d}{dt} _{|t=0} f( (x,z) \star (t e_i,0 ) ) = \left( \partial_{x_i} - \sum_{j=1, j\neq i}^n \frac{1}{2} x_j \partial _{z_{i,j}} \right) f(x,z) 
 \]
 and the vertical vector fields  for $1\leq i<j \leq n$ by
 \[
Z_{i,j} (f) (x,z)= \frac{d}{dt} _{|t=0} f( (x,z) \star (0,t e_i\symp e_j ) ) = \partial_{z_{i,j}} f(x,z)
 \]
with $z=\sum_{1\leq i<j\leq n } z_{i,j}e_i \symp e_j$ and where in  the definition of $X_i$, if $i>j$, we set  $\partial _{z_{i,j}}=- \partial _{z_{j,i}}$.

 It is also possible to define the Carnot-Carath\'eodory subRiemmanian distance  on $\Ge_n$ by :
\[
d_{\Ge_n} (g,g')= \inf_{\gamma}  \int_0 ^1  | \dot \gamma(t)  |_{\frak h} dt
\]
where $\gamma$ ranges over the horizontal curves connecting $\gamma(0)=g$ and $\gamma(1)=g'$; i.e.,  absolutely continuous curves such that  $\dot \gamma (t) \in \Span \{ X_i(\gamma(t)), 1\leq i\leq n\}$ almost surely and where 
$|\cdot  |_{\frak h}$ is a Euclidean norm on $\Span \{ X_i(\gamma(t)), 1\leq i\leq n\}$  obtained by asserting that $(X_1,\dots, X_n)$ is an orthonormal basis in each point.  
As for  the Heisenberg group, these Carnot groups  admits homogeneous dilations adapted both to the distance and the group structure given by 
  \[
  \dil_\lambda (x,z)= (\lambda x, \lambda^2 z).
  \]
Finally the horizontal gradient $\nabla_{\frak h} f$ is $\sum_{i=1}^n X_i(f) X_i$ 
whereas  the vertical gradient is defined by $\nabla_{\frak v} f= \sum_{1\leq i<j \leq n} Z_{i,j} (f) Z_{i,j}$.

The total variation estimate implies the following $L^\infty$ gradient bounds. 

\begin{cor}\label{C2}
    	Let $\Ge_n$ be the free step 2 Carnot group of   of rank $n\geq 2$. For any bounded measurable  function   $f$ on $\Ge_n$, for any    $g\in \Ge_n$ and $t>0$,
	\begin{equation}\label{gradientInequality-h2}
	\Vert \nabla_{\frak h }P_t f(g) \Vert  \leq   \frac{2 C_1(n)}{\sqrt{t}} ||f||_{\infty}
	\end{equation}
	and 
	\begin{equation}\label{gradientInequality-v2}
	\Vert \nabla_{\frak v }P_t f(g) \Vert  \leq   \frac{2  \sqrt 2 C_2(n)}{t} ||f||_{\infty}
	\end{equation}
	where $C_1(n)$ and $C_2(n)$ are the constant appearing in Theorem \ref{T2}  (or \ref{T1}).
\end{cor}

\begin{proof}
The proof is standard. Let $f$ be a bounded measurable  function  on $\Ge_n$ and let $g,\tg\in \Ge_n$.
\begin{align}\label{inegalité}
	|P_tf(g)-P_tf(\tg)|&=\left|\esp\left[f(\bbb_t^g)-f(\bbb_t^{\tg} )\right]\right| \notag\\
	&= \left|\esp\left[f(\bbb_t^g)-f(\bbb_t^{\tg} ) \mathbb{1}_{\{\bbb_t^g \neq \bbb_t^{\tg}\}}\right]\right| \notag\\
	&\leq 2||f||_{\infty}\,  \pr\left(\bbb_t^g \neq \bbb_t^{\tg} \right).
	\end{align}
Now
 since there exists a constant $C>0$ such that 
 \[
 \Vert \tilde x- x \Vert  \leq d_{CC}(g,\tg)  \textrm{ and } \Vert \zeta \Vert =\|\tilde z -z -\frac12 x\symp \tilde x\| \leq C  d_{CC}(g,\tg)^2, 
 \]
by  Theorem \ref{T2} (or Theorem \ref{T1}), one can construct a coupling of 
$\bbb_t^g$  and $\bbb_t^{\tg}$ such that 
\[
\pr\left(\bbb_t^g \neq \bbb_t^{\tg} \right) \leq \frac{C_1(n)}{\sqrt t} d_{CC}(g,\tg) + \frac{C C_2(n)}{t}  d_{CC}(g,\tg)^2.
\]
Dividing by $d_{CC}(g,\tg)$ and letting $\tilde g \to g$ gives the horizontal gradient inequality $\eqref{gradientInequality-h}$.
When $\tilde x=x$, the above estimate writes:
\[
\pr\left(\bbb_t^g \neq \bbb_t^{\tg} \right) \leq \frac{C_2(n)}{ t} \Vert \tilde z -z\Vert
\]
and  the vertical gradient inequality \eqref{gradientInequality-v} follows in a similar way.
\end{proof}

\subsection{Coupling with change of probability: application to reverse Sobolev inequalities}

In this section we will construct couplings at time $T$ with probability one, but the price to pay will be to make changes of probabilities for the second process. The distance between semigroups will be measured by the change of probability. The main results are a log Harnack inequality (Theorem~\ref{T-LH}), an integration by parts formula (Theorem~\ref{T-B}) 
for  the spatial derivative $dP_Tf$ of the semigroup $P_Tf$ of the Brownian motion
and  reverse Poincar\'e or  Sobolev inequalities (Theorem~\ref{T-RP} and Corollary \ref{C3}) and some estimates of the gradient of the heat kernel (Corollary \ref{Prop: HeatKernel}).

The notations are the same as in the previous section. The processes $(\bbb^g_t)_t:=(X_t,z_t)_t$ and $(\bbb^{\tg}_t)_t:=((\tilde X_t,\tilde z_t))_t$ started respectively at $g=(x,z)$ and $\tilde g=(\tilde x,\tilde z)$  are defined with Equations~\eqref{E22}, \eqref{E21} and \eqref{E23}. The sequence $(\xi_k)_{k\ge 0}$ will be identically distributed will law $\cN(0,I_n)$ under the probability $\P$. The difference will be that we will look for a sequence $(\tilde\xi_k)_{k\ge 0}$ which is independent and identically distributed with law $\cN(0,I_n)$ under another probability $\P(\tilde g)$, and so that at time~$T$, a.s. $\bbb^{g}_T=\bbb^{\tg}_T$.

Fix $K\in \{n+1,\ldots\}\cup\{\infty\}$ and let 
\begin{equation}
    \label{E44.1}
    J_K:=\{\ell\in \N,\  \ell\le K\}\quad\hbox{if}\quad K<\infty,\quad J_\infty:=\N\quad \hbox{and}\quad J_K^\ast:=J_K\backslash\{0\}\ \forall \ K.
\end{equation}
We will take 
\begin{equation}
    \label{E44}
    \xi_k=\tilde \xi_k\quad\hbox{for all}\quad 
    k\not\in 3J_K
\end{equation}
so that we are left to couple 
\begin{equation}
    \label{E45}
    (\xi_\ell,\tilde \xi_\ell),\quad \ell\in 3J_K.
\end{equation}
Now we consider the  sequence $(V_k)_{k\in J_K^\ast}$ defined in \eqref{E34}, of independent random vectors taking their values in $\R^n$, with the same law $\cN(0,I_n)$. We solve in $(U_k)_{k\in J_K^\ast}$ the equation 
\begin{equation}
    \label{E46}
    \sum_{k\in J_K^\ast} U_k\symp V_k=W
\end{equation}
with $W$ given by Equation~\eqref{E31}. Then we will choose $(\tilde\xi_k)_{k\ge 0} $ such that almost surely
\begin{equation}
    \label{E47}
    \tilde \xi_0-\xi_0=\frac{x-\tilde x}{\sqrt{T}}=:\hat U_0
\end{equation}
and 
\begin{equation}
    \label{E48}
    \forall\ k\in J_K^\ast,\  \tilde \xi_{3k}-\xi_{3k}=\frac{U_k}{T\sqrt{\alpha_{3k}^2+\alpha_{3k-1}^2}}=:\hat U_k.
\end{equation}

We denote
\begin{equation}
\label{E51}
  \forall k\in J_K^\ast,\quad    V_k=\sum_{i= 1}^n V_k^i e_i,\quad U_k=\sum_{j= 1}^nU_k^j e_j,
\end{equation}
\begin{equation}
    \label{E50.1}
    \beta_k=\frac{1}{T\sqrt{\alpha_{3k}^2+\alpha_{3k-1}^2}}
\end{equation}
\begin{equation}
    \label{E52}
    \hat\SV=\hat\SV_K=\left(\frac{V_k^i}{\beta_k}\right)_{1\le i\le n,\ k\in J_K^\ast}, \quad \hat\SU=\hat\SU_K=(\beta_k U_k^i)_{1\le j\le n, \ k\in J_K^\ast},
\end{equation}
the upper index representing the lines and the lower index representing the columns. 
With these notations and similarly  as before, Equation~\eqref{E46} is equivalent to 
\begin{equation}
\label{E53-equivalence}
  \hat \SU  \hat \SV^t -   \hat\SV \hat\SU^t=\SW 
\end{equation}
 with $\SW$  defined by \eqref{E15}.
In particular, we have a solution of Equation~\eqref{E46} if 
\begin{equation}
\label{E53}
    \hat\SV \hat\SU^t=-\frac 1 2\SW.
\end{equation}
The $n\times n$ matrix 
\begin{equation}
    \label{E54}
    \hat\SV \hat\SV^t=\sum_{k\in J_K^\ast} \frac1{\beta_k^2}V_kV_k^t
\end{equation}
is a.s. well-defined since $\displaystyle \E\left[\sum_{k=1}^\infty \frac1{\beta_k^2}\tr(V_kV_k^t)\right]<\infty$ (the computation~\eqref{E42} proves that $\beta_k$ is of order $k$). It is a.s. symmetric positive since $K\ge n$. Consequently it is invertible, and a solution to~\eqref{E53} is given by 
 \begin{equation}
     \label{E55}
     \hat\SU^t=-\frac{1}{2}\hat\SV^t(\hat\SV \hat\SV^t)^{-1}\SW.
 \end{equation}

Let us make a specific choice of probability space, which will be very convenient for our computations.
This probability space is $(\Omega,\SA,\P)$, where $\Omega:=\ell^2(\R^n)$ is the Hilbert space  of square integrable $\R^n$-valued sequences, $\SA$ is the smallest $\sigma$-field for which the projections are measurable, completed with respect to the probability measure $\P$ for which the canonical projections
\begin{align*}
\xi_k :\Omega&\to \R^n\\
\omega=(\omega_0,\omega_1,\ldots,\omega_k,\ldots)&\mapsto \omega_k=:\xi_k(\omega)
\end{align*}
 are i.i.d. and $\cN(0,I_n)$. We will need to split $\Omega$ into two supplementary orthogonal spaces: $\Omega=\Omega_a\oplus\Omega_b$. Let us now describe these spaces. For $k\ge 1$ and $1\le i\le n$, we denote by $e_k^i$ the element of $\Omega$ which satisfies $\xi_\ell(e_k^i)=0$ if $\ell\not=k$ and  $\xi_k(e_k^i)=e_i$, the $i$-th element of the canonical basis of $\R^n$. Letting $(f_1,\ldots,f_n)$ be an orthonormal basis of $\R^n$ such that $\|x-\tilde x\|_2f_1=x-\tilde x$, for $1\le i\le n$ we denote by $f_0^i$ the element of $\Omega$ such that $\xi_\ell(f_0^i)=0$ if $\ell\not=0$ and $\xi_0(f_0^i)=f_i$. Notice that the $(e_k^i),\  k\ge 1,\ 1\le i\le n$ together with the $(f_0^i), \ 1\le i\le n$ form an Hilbertian basis of $\Omega$ and that the random variables $\langle e_k^i, \omega\rangle$, $\langle f_0^i, \omega\rangle$  are i.i.d and $\SN(0,1)$. Define 
 \begin{equation}
     \label{E55.1}
     \Omega_a={\Span}\left\{f_0^1,\ e_{k}^i, \  k\in 3J_K^{\ast} , \ 1\le i\le n\right\},
 \end{equation}
\begin{equation}
     \label{E55.2}
     \Omega_b=\Omega_a^\perp={\Span}\left\{f_0^i,\ 2\le i\le n\right\}\oplus{\Span}\left\{\ e_{k}^i, k\notin 3J_K, \ 1\le i\le n\right\}.
 \end{equation}
 
For the sequel, we will denote $\omega_a$ (resp. $\omega_b$) the projection of $\omega$ on $\Omega_a$ (resp. $\Omega_b$). 
 Let  $\SA_a$ and $\SA_b$ be  the canonical $\sigma$-fields and  $\P_a$ (resp. $\P_b$) be  such that the $\langle \omega_a,e_k^i\rangle$, $k\in 3 J_K^\ast, 1\le i \le n $,  $\langle \omega_a,f_0^1\rangle$ (resp. $\langle \omega_b,e_\ell^j\rangle$, $\langle \omega_b,f_0^i\rangle$ $\ell \notin 3J_K, 1\leq j \leq n, 2\leq i \leq n$) are independent $\SN(0,1)$ random variables. Then
 \begin{equation}
     \label{E55.3}
     \begin{split}
         \left(\Omega_a\times \Omega_b,\SA_a\times\SA_b,\P_a\times \P_b\right)&\to (\Omega,\SA,\P)\\
         (\omega_a,\omega_b)&\mapsto \omega_a+\omega_b
     \end{split}
 \end{equation}
 is an isometry.
 
 Recall that  $\tilde\xi_k=\xi_k$ if $k\not\in 3J_K$ and $\tilde \xi_{3k}=\xi_{3k}+\hat U_k$ if $k\in J_K$. 
 Let $\P(\tilde g)$ be the probability on $\Omega$ such that all $\tilde\xi_k$ are i.i.d. and $\cN(0,1)$. 
\begin{lemme}
\label{L3}
The probability $\P(\tilde g)$ is equivalent to $\P$, and 
\begin{equation}
    \label{E56}
   R(u)(\omega):= \frac{d\P(\tilde g)}{d\P}(\omega)=
   e^{-\left\langle\omega,u\right\rangle-\frac12\|u\|^2}
\end{equation}
where $u=u(\tilde g)(\omega)\in \Omega$ is defined by
\begin{equation}
    \label{E57}
    u_k=0\ \forall \ k\not\in 3J_K\quad \hbox{and}\quad u_{3k}=\hat U_k(\omega) \ \forall k\in J_K,
\end{equation}
and $\displaystyle \langle \omega,u\rangle =\sum_{k=0}^\infty \langle \omega_k, u_k\rangle_{\R^n}$.

In particular, 
\begin{equation}
    \label{E58}
    dR(u(\cdot))|_{\tilde g=g}=-\sum_{k=0}^\infty\left\langle \xi_{3k},  d\hat U_k(\cdot)|_{\tilde g=g}\right\rangle.
\end{equation}
Moreover, for all measurable $F:\Omega\to\R$, we have that $F$ is $\P$-integrable if and only if $\omega\mapsto F(\omega+u(\omega))$ is $R(u)\cdot\P$-integrable, and in this case
\begin{equation}
    \label{E57.1}
    \E[F(\omega)]=\E\left[F(\omega+u(\omega))R(u(\omega)) (\omega)\right].
\end{equation}
We also have \begin{equation}
    \label{E57.2}
    \E[F(\omega-u(\omega))]=\E\left[F(\omega)R(u(\omega))(\omega)\right].
\end{equation}
\end{lemme}
\begin{proof}
First observe that for a fixed deterministic nonzero vector $u\in \Omega_a$, we can make the orthogonal decomposition
\begin{equation}
    \label{E58.1}
    \omega_a=\left\langle \omega_a,\frac{u}{\|u\|}\right\rangle\frac{u}{\|u\|}+P_{(u)^\perp}^{\Omega_a}(\omega_a)
\end{equation}
where $\displaystyle \left\langle \omega_a,\frac{u}{\|u\|}\right\rangle$ is an $\cN(0,1)$ real-valued random variable independent of $P_{(u)^\perp}^{\Omega_a}(\omega_a)$.
Now remark that the random vector $u(\omega)$ satisfies $u(\omega)=u(\omega_b)$ in the decomposition $\omega=\omega_a+\omega_b$ of~\eqref{E55.3}. This is due to the fact that the $\hat U_k$ do not change when one replaces $\xi_0$ by $\xi_0-\langle\xi_0,f_1\rangle f_1$ in the expression of 
$$
W=-\zeta +\|x-\tilde x\|_2f_1\symp \left(\frac{\sqrt T}2 \xi_0-\sqrt{T}\alpha_0\tilde \xi_1\right).
$$
In other words, $u$ is measurable  with respect to $\sigma$-field $\SG:=\sigma(\xi_k, k\not\in 3J_K^\ast)\vee\sigma(P_{(x-\tilde x)^\perp}^{\R^n}(\xi_0))$ ($P_{(x-\tilde x)^\perp}^{\R^n}$ denoting the projection in $\R^n$ orthogonal to $x-\tilde x$). 

 A second important fact is that $\omega\mapsto u(\omega)$ takes its values in $\Omega_a$.
In other words $ u_\ell=0$ if $\ell\not\in 3J_K^\ast$ and $ u_{0}$ is collinear to $x-\tilde x$. Consequently, 
conditioned to $\SG$, $u$ is a $\Omega_a$-valued constant. So we can make the same decomposition as in~\eqref{E58.1}:
\begin{align*}
\omega_a=\left\langle \omega_a,\frac{u(\omega)}{\|u(\omega)\|}\right\rangle\frac{u(\omega)}{\|u(\omega)\|}+P_{(u(\omega))^\perp}^{\Omega_a}(\omega_a)
\end{align*}
where conditioned to $\SG$, $\left\langle \omega_a,\frac{u(\omega)}{\|u(\omega)\|}\right\rangle$ is an $\SN(0,1)$ random variable independent of $P_{(u(\omega))^\perp}^{\Omega_a}(\omega_a)$. Adding $\omega_b$ which is $\SG$-measurable and orthogonal to $\Omega_a$ we get 
\begin{equation}
    \label{E58.2}
    \omega=\left\langle \omega,\frac{u}{\|u\|}\right\rangle\frac{u}{\|u\|}+P_{(u)^\perp}(\omega)
\end{equation}
where conditioned to $\SG$, $\displaystyle \left\langle \omega,\frac{u}{\|u\|}\right\rangle$ is an $\cN(0,1)$ real-valued random variable independent of $P_{(u)^\perp}(\omega)$.

Let $F : \Omega\to\R$ a bounded measurable function. 
\begin{align*}
    \E[F(\omega)]&=\E\left[\left.\E\left[F\left(\left\langle \omega,\frac{ u}{\| u\|}\right\rangle\frac{ u}{\| u\|}+P_{( u)^\perp}(\omega)\right)\right|P_{( u)^\perp}(\omega),\ \SG\right]\right]\\
    &=\E\left[\int_\R F\left(x\frac{ u}{\| u\|}+P_{( u)^\perp}(\omega)\right)\varphi(x)\,dx\right]
\end{align*}
where $\varphi$ is the density of $\cN(0,1)$. 
But
$$
\int_\R F\left(x\frac{ u}{\| u\|}+P_{( u)^\perp}(\omega)\right)\varphi(x)\,dx=\int_\R F\left((x+\| u\|)\frac{ u}{\| u\|}+P_{( u)^\perp}(\omega)\right)\varphi(x+\| u\|)\,dx
$$
yielding
\begin{align*}
    &\E[F(\omega)]\\&=\E\left[\int_\R F\left((x+\| u\|)\frac{ u}{\| u\|}+P_{( u)^\perp}(\omega)\right)\varphi(x+\| u\|)\,dx\right]\\
    &=\E\left[\int_\R F\left((x+\| u\|)\frac{ u}{\| u\|}+P_{( u)^\perp}(\omega)\right)\frac{\varphi(x+\| u\|)}{\varphi(x)}\varphi(x)\,dx\right]\\
    &=\E\left[
    \E\left[
    \frac{\varphi\left(\left\langle \omega,\frac{ u}{\| u\|}\right\rangle+\| u\|\right)}{\varphi\left(\left\langle \omega,\frac{ u}{\| u\|}\right\rangle\right)}F\left(\left(\left\langle \omega,\frac{u}{\| u\|}\right\rangle+\| u\|\right)\frac{ u}{\| u\|}+P_{( u)^\perp}(\omega)\right)
    |P_{( u)^\perp}(\omega),\ \SG\right]
    \right]
    \end{align*}
    recalling that conditioned to $\SG$, $\displaystyle \left\langle \omega,\frac{u}{\|u\|}\right\rangle$ is an $\cN(0,1)$ real-valued random variable independent of $P_{(u)^\perp}(\omega)$. So
    \begin{align*}
       &\E[F(\omega)]\\ 
    &=\E\left[
    \frac{\varphi\left(\left\langle \omega,\frac{ u}{\| u\|}\right\rangle+\| u\|\right)}{\varphi\left(\left\langle \omega,\frac{ u}{\| u\|}\right\rangle\right)}F\left(\left(\left\langle \omega,\frac{u}{\| u\|}\right\rangle+\| u\|\right)\frac{ u}{\| u\|}+P_{( u)^\perp}(\omega)\right)
    \right]\\
    &=\E\left[\frac{\varphi\left(\left\langle \omega,\frac{ u}{\| u\|}\right\rangle+\| u\|\right)}{\varphi\left(\left\langle \omega,\frac{ u}{\| u\|}\right\rangle\right)}F\left(\omega+ u\right)\right].
\end{align*}
Observing that 
$$
\frac{\varphi\left(\left\langle \omega,\frac{ u}{\| u\|}\right\rangle+\| u\|\right)}{\varphi\left(\left\langle \omega,\frac{ u}{\| u\|}\right\rangle\right)}=e^{-\left\langle\omega, u\right\rangle-\frac12\|u\|^2}
$$
yields~\eqref{E56} via~\eqref{E57.1}. Equation~\eqref{E58} is a direct consequence. Finally, observe that $u(\omega-u(\omega))=u(\omega)$ since $u(\omega)=u(\omega_b)$ and $u(\omega)\in \Omega_a$. Equation \eqref{E57.2} is then obtained from \eqref{E57.1}.
\end{proof}
\begin{cor}
\label{C1}
Take $K=2n+1$.
Let $R=R(u)$ be as in Lemma~\ref{L3}. Then $R\ln R$ is integrable and 
\begin{equation}
    \label{E58.4}
    \begin{split}
    &\E[R\ln R]=\frac12\E\left[\|u\|^2\right]\\&\le
    \frac{\left\|x-\tilde x\right\|_2^2}{2T}+
    \left(6\sqrt{n}+\frac{4}{\sqrt{n}}\right)^2\left(\frac1{T^2}\left\|z-\tilde z-\frac12x\symp \tilde x\right\|^2 +\frac{2(n-1)}{3T}\left\|x-\tilde x\right\|^2_2\right).
    \end{split}
\end{equation}
\end{cor}
\begin{proof}
Recall that $\|u\|^2=\|\hat U_0^2\|_2+\|\hat U\|^2=\frac{\|x-\tilde x\|_2^2}{T}+\|\hat U\|^2$, $\hat U$ being defined as in~\eqref{E35.1} with $m=K$.
First observe that the inequality in~\eqref{E58.4} comes from~\eqref{E42.1} and~\eqref{E42.3}.

Using~\eqref{E57.2} with $F(\omega):=\ln(R(u(\omega)))(\omega)$,
\begin{align*}
 & \E\left[R(u(\omega))(\omega)\ln R(u(\omega))(\omega)\right]\\& =\E\left[\ln R(u(\omega-u(\omega)))(\omega-u(\omega))\right]\\
   &=\E\left[\ln R(u(\omega))(\omega-u(\omega))\right]\\
   &=\E\left[-\langle \omega-u(\omega),u(\omega)\rangle-\frac12\|u(\omega)\|^2 \right]\\
   &=\E\left[-\E[\langle \omega,u(\omega)\rangle|\SG]+\frac12\|u(\omega)\|^2 \right]\quad\hbox{with $\SG$ defined in the proof of Lemma~\ref{L3}} \\
   &=\E\left[\frac12\|u(\omega)\|^2 \right]
\end{align*}
since $\E[\langle \omega,u(\omega)\rangle|\SG]=0$: $u(\omega)$ is $\SG$-measurable and conditioned to $\SG$ $\langle \omega,u(\omega)\rangle$ is Gaussian and centered.
\end{proof}
In the sequel, we will need  the solution  $\hat \SU$ defined by \eqref{E55} to have moments of any order. To get this integrability condition, we will have to consider  the case $K=+\infty$.

We  first set  two preparatory lemmas.

\begin{lemme}
\label{L5}Let $h>0$.
Let $(Y_\ell)_{\ell\ge 1}$ be a sequence of independent  gamma distributed real-valued random variables 
with the same  parameter $h$ and  define 
\begin{equation}\label{ESh}
   S_h:= {\frac{2}{\pi^2}}\sum_{\ell=1}^\infty \frac{Y_\ell}{\ell^2}.
\end{equation}
For $a>0$, one has 
\begin{equation}
    \label{E58.6bis}
   \E\left[S_h^{-a}\right]=2^{1+h-a}\frac{\Gamma(2a+h)}{\Gamma(h)\Gamma(a)}\sum_{n=0}^\infty \frac{\Gamma(n+h)}{\Gamma(n+1)}\frac1{(2n+h)^{2a+h}}.
\end{equation}
In particular, we have
\begin{equation}
    \label{E58.3}
    \E\left[S_1^{-a}\right]\le \frac{(4a+1)\Gamma(2a+1)}{{2^{a}}\Gamma(a+1)}
\end{equation}
and 
\begin{equation}\label{E58.3bis}
 \E\left[S_{\frac 1 2} ^{-\frac 1 2}\right] \leq 2 \sqrt 2 +\frac{\sqrt 2}{2}.
\end{equation}
\end{lemme}
\begin{proof}
The Laplace transform of $S_h$ 
is given by 
\begin{equation}
    \label{E58.5}
  \forall \ \lambda >0, \quad \E\left[e^{-\lambda S_h}\right]=\left(\frac{\sqrt{2\lambda}}{\sinh\sqrt{2\lambda}}
  \right)^h,
\end{equation}
see~\cite{BPY01}. 
 On the other hand, making the change of variable $u=S_h\lambda$ in the following integral gives
\begin{align*}
     \E\left[\int_0^\infty \lambda^{a-1}e^{-\lambda S_h}\, d\lambda\right]&=\E\left[S_h^{-a}\int_0^\infty u^{a-1}e^{-u}\, d\lambda\right]=\E\left[S_h^{-a}\right]\Gamma(a).
\end{align*}
From this we obtain
\begin{align*}
   \E\left[S_h^{-a}\right]&=\frac1{\Gamma(a)} \E\left[\int_0^\infty \lambda^{a-1}e^{-\lambda S_h}\, d\lambda\right]
   =\frac1{\Gamma(a)} \int_0^\infty \lambda^{a-1}\E\left[e^{-\lambda S_h}\right]\, d\lambda\\
   &=\frac1{\Gamma(a)} \int_0^\infty \lambda^{a-1} \left(\frac{\sqrt{2\lambda}}{\sinh\sqrt{2\lambda}}\right)^h\, d\lambda
   =
   \frac{(2\sqrt{2})^h}{\Gamma(a)}\int_0^\infty \lambda^{a+\frac h 2-1 }\frac{1}{e^{h\sqrt{2\lambda}}\left(1-e^{-2\sqrt{2\lambda}}\right)^h}\, d\lambda\\
   &= \frac{(2\sqrt{2})^h}{\Gamma(a)\Gamma(h)} \sum_{n=0}^\infty \frac  {\Gamma(n+h)}{\Gamma(n+1)}
   \int_0^\infty \lambda^{a+\frac h 2-1 }  e^{-(2n+h) \sqrt{2\lambda}}\, d\lambda,
\end{align*}
by Fubini theorem and since for $h>0$ and $|x|<1$, 
\[ 
\frac{1}{(1-x)^h}= \frac 1  {\Gamma(h)} \sum_{n\geq 0} \frac  {\Gamma(n+h)}{\Gamma(n+1)}x^n. 
\]
Making the change of variable $u=(2n+h)\sqrt{2\lambda}$ yields
\begin{equation*}
  \int_0^\infty \lambda^{a+\frac h 2-1 }  e^{-(2n+h) \sqrt{2\lambda}}\, d\lambda =\frac{\Gamma(2a+h)}{2^{a+\frac h 2 -1}(2n+h)^{2a+h}} 
\end{equation*}
and \eqref{E58.6bis} follows.

In particular for $h=1$ we have
\begin{equation}
    \label{E58.6}
   \E\left[S_1^{-a}\right]=2^{2-a}\frac{\Gamma(2a+1)}{\Gamma(a)}\sum_{n=0}^\infty\frac1{(2n+1)^{2a+1}}. 
\end{equation}
Now 
\[
\sum_{n=0}^\infty\frac1{(2n+1)^{2a+1}}\le 1+\int_0^\infty\frac{dx}{(2x+1)^{2a+1}}=1+\frac{1}{4a}
\]
and \eqref{E58.3} follows.
For  $h=1/2$ and  $a=1/2$, one have 
\begin{equation}
    \label{E58.6bis2}
   \E\left[S_\frac{1}{2}^{-\frac{1}{2}}\right]=2\frac{\Gamma(\frac{3}{2})}{\Gamma(\frac{1}{2})^2} \sum_{n=0}^\infty \frac{\Gamma(n+\frac{1}{2})}{\Gamma(n+1)}\frac1{(2n+\frac{1}{2})^{2a+\frac{1}{2}}}.
\end{equation}
Since for  $n\geq 1$
\[
\frac{\Gamma(n+\frac 1 2 )}{\Gamma(n+1)} \leq \frac{\Gamma(\frac 3 2)}{\Gamma(2)}=\frac {\Gamma\left(\frac 1 2\right)} 2 = \frac{\sqrt \pi}{2}
\textrm{ and } \frac{\Gamma(\frac 32 )}{\Gamma(\frac 1 2)^2}=\frac 12 \frac1 {\Gamma(\frac 1 2)}=
\frac{1}{2\sqrt \pi},
\]
one has 
\[
 \E\left[S_{\frac 1 2} ^{-\frac 1 2}\right] \leq 2^{3/2}+ \frac{1}{2} \sum_{n=1}^\infty \frac1{(2n+\frac{1}{2})^{3/2}} 
\le 2^{3/2} + \frac{1}{2} \int_0^\infty\frac{dx}{(2x+\frac{1}{2})^{3/2}}=2 \sqrt 2 +\frac{\sqrt 2}{2},
\]
which ends the proof of Lemma \ref{L5}.
\end{proof}

\begin{lemme}
\label{L4}
Let $(V_k)_{k\ge 1}$ be a sequence of $\R^n$-valued independent random variables with law $\cN(0,I_n)$. 
Then for any $a>0$, 
\begin{equation}
     \label{E-moment}
   \E\left[\tr\left(\left(\sum_{k=1}^\infty \frac{V_k V_k^t}{\beta_k^2}\right)^{-1}\right)^{a} \right]
   \leq \frac{(C_3(n))^{a}}{T^{2a}} \frac{(4a+1)\Gamma(2a+1)}{\pi^{2a}\Gamma(a+1)}.
 \end{equation}
with $C_3(n)= 8n^2(3n+4)^2$.
Moreover, for any $p\in(0,1)$ and all $\lambda>0$,we have
\begin{equation}
    \label{E59}
    \E\left[\exp\left(\lambda \tr\left(\left(\sum_{k=1}^\infty \frac{V_k V_k^t}{\beta_k^2}\right)^{-1}\right)^{p} \right)\right]
    \le 1+\sum_{q=1}^\infty \left(\frac{(C_3(n))^p}{(T\pi)^{2p}}\lambda\right)^q \frac{(4pq+1)\Gamma(2pq+1)}{q!\Gamma(pq+1)}<\infty.
\end{equation}
\end{lemme}

Note that here, in Lemma \ref{L4}, the estimates do not seem optimal in term of the dimension $n$.

\begin{proof}[Proof of Lemma \ref{L4}]
 We have 
 $$
 \sum_{k=1}^\infty \frac{V_k V_k^t}{\beta_k^2}\ge \sum_{\ell=1}^\infty \frac1{\beta_{\ell(n+1)}^2}\SM_\ell\quad \hbox{with}\quad \SM_\ell:=\sum_{\ell'=(\ell-1)(n+1)+1}^{\ell(n+1)}V_{\ell'}V_{\ell'}^t.
 $$
 The matrices $\SM_\ell$ are Wishart $\SW(n,n+1)$ with smallest eigenvalue $\lambda_{\min}(\SM_\ell)$ having an exponential law with parameter $n/2$, or equivalently a law $\frac1n\chi^2(2)$. Consequently, by independence, we have 
 \begin{equation}
     \label{E61-0}
    \lambda_{\min}\left(\sum_{k=1}^\infty \frac{V_k V_k^t}{\beta_k^2}\right) \ge \sum_{\ell=1}^\infty \frac{Y_\ell}{n\beta_{\ell(n+1)}^2}
 \end{equation}
with $Y_{\ell}$ independent $\chi^2(2)$ random variables. 
Then using $\beta_k\le \frac{2\sqrt{2}(3k+1)}{T}$ 
we can write
\begin{equation}
     \label{E61-1}
    \lambda_{\min}\left(\sum_{k=1}^\infty \frac{V_k V_k^t}{\beta_k^2}\right) \ge T^2\sum_{\ell=1}^\infty \frac{n}{C_3(n)\ell^2}Y_\ell
 \end{equation}
 with $C_3(n)=n^2\left(2\sqrt{2}(3n+4)\right)^2$. We have 
 \begin{align*}
 \tr\left(\left(\sum_{k=1}^\infty \frac{V_k V_k^t}{\beta_k^2}\right)^{-1}\right)&\le n\lambda_{\max}\left(\left(\sum_{k=1}^\infty \frac{V_k V_k^t}{\beta_k^2}\right)^{-1}\right)\\
 &=n\left(\lambda_{\min}\left(\sum_{k=1}^\infty \frac{V_k V_k^t}{\beta_k^2}\right)\right)^{-1}
 \end{align*}
 
 Consequently, for $a>0$,
 \begin{equation}
     \label{E61-3}
   \E\left[\tr\left(\left(\sum_{k=1}^\infty \frac{V_k V_k^t}{\beta_k^2}\right)^{-1}\right)^{a} \right]\le \frac{ C_3^{a}(n)}{T^2} \E\left[\left(\sum_{\ell=1}^\infty \frac{Y_\ell}{\ell^2}\right)^{-a}\right].
 \end{equation}
 The estimate \eqref{E-moment} thus directly follows from Lemma \ref{L5}.
 We now turn to the exponential moments. Let $0<p<1$ and $\lambda >0$,  we have
 \begin{align*} 
     \E\left[\exp\left(\lambda\left(\tr\left(\sum_{k=1}^\infty \frac{V_k V_k^t}{\beta_k^2}\right)^{-1}\right)^{p}\right)\right]
     &\le  \E\left[\exp\left(\left(\frac{C_3(n)}{T^2}\right)^p\lambda\left(\sum_{\ell=1}^\infty \frac{Y_\ell}{\ell^2}\right)^{-p}\right)\right]\\
     &=1+\sum_{q=1}^\infty \frac{(C_3(n))^{pq}\lambda^q}{q!T^{2pq}}\E\left[\left(\sum_{\ell=1}^\infty \frac{Y_\ell}{\ell^2}\right)^{-pq}\right]\\
    & \le 1+\sum_{q=1}^\infty \frac{(C_3(n))^{pq}\lambda^q}{q!T^{2pq}}\frac{(4pq+1)\Gamma(2pq+1)}{\pi^{2pq}\Gamma(pq+1)}
    \end{align*}
    where we used Lemma~\ref{L5} with $a=pq$.
    This is exactly the first inequality in~\eqref{E59}. We are left to prove that the right hand side in \eqref{E59} is finite. Using $\ln\Gamma(a)\sim a\ln(a)$ as $a\to\infty$ we get
    $$
    \ln\left(\frac{(4pq+1)\Gamma(2pq+1)}{q!\Gamma(pq+1)}\right)\sim
    \left(2pq-q-pq\right)\ln(q)=q(p-1)\ln(q)<-\varepsilon q\ln(q)
    $$
    with $\varepsilon=\frac{1-p}2$. Letting $\alpha=\left(\frac{C_3(n)}{\pi^2}\right)^p\lambda$ we have $$
    \left(\left(\frac{C_3(n)}{\pi^2}\right)^p\lambda\right)^q \frac{(4pq+1)\Gamma(2pq+1)}{q!\Gamma(pq+1)}\le \alpha^q q^{-\varepsilon q}\quad\hbox{for $q$ sufficiently large}
    $$
    and $\sum_{q=1}^\infty \alpha^q q^{-\varepsilon q}<\infty$, proving the finiteness of the right hand side of~\eqref{E59}.
\end{proof}

\medbreak

 After these preliminary results, we  now turn to the analytic consequence  for the semi-group of this change of probability method.
Let $f : \Ge_n \to \R$ a bounded measurable function. We recall that 
\begin{equation}
    \label{E62}
    P_Tf(g)=\E[f(\bbb^{g}_T)]\quad \hbox{together with}\quad P_Tf(\tilde g)=\E[f(\bbb^{\tg}_T) R(u)];
\end{equation}
$\tilde g$ and $u$ being related as in Lemma~\ref{L3}.
But with our construction, we have a.s $\bbb^{\tg}_T= \bbb^{g}_T$, yielding
\begin{equation}
    \label{E63}
    P_Tf(\tilde g)=\E[f(\bbb^{g}_T)R(u)].
\end{equation}
From this and Corollary~\ref{C1} we get the following log Harnack inequality.
\begin{thm}
\label{T-LH}
Let $f$ be a positive function in $\Ge_n$, $T>0$ and $g=(x,z),\tilde g=(\tilde x,\tilde z)\in \Ge_n$. Then
\begin{equation}
    \label{E63.1}
    \begin{split}
    &P_T(\ln f)(\tilde g)\le \ln (P_Tf(g))\\&+
    \frac{\left\|x-\tilde x\right\|^2_2}{2T}+
    \left(6\sqrt{n}+\frac{4}{\sqrt{n}}\right)^2\left(\frac1{T^2}\left\|z-\tilde z-\frac12x\symp \tilde x\right\|^2 +\frac{2(n-1)}{3T}\left\|x-\tilde x\right\|^2_2\right).
    \end{split}
\end{equation}
\end{thm}
\begin{proof}
Again take $K=2n+1$. By Equation~\eqref{E62} applied to $\ln f$ and Young inequality,
\begin{align*}
    P_T(\ln f)(\tilde g)&=\E[\ln f(\bbb^{g}_T)R(u)]\\
    &\le \E[R(u)\ln R(u)]+\ln\E\left[\exp\ln f(\bbb_T^g)\right]\\
    &=\E[R(u)\ln R(u)]+\ln (P_Tf(g)).
\end{align*}
We conclude with~\eqref{E58.4}.
\end{proof}
The next theorem aims at establishing an integration by parts formula for the derivative of the semi-group.
\begin{thm}
\label{T-B}
Fix $K\ge n+2$ (and possibly infinite).
Let $f : \Ge_n \to \R$ be a bounded continuous function, $g=(x,z), h=(h_x,h_z) \in \Ge_n$. Denote $\tilde g= g+h$  we have
\begin{equation}
    \label{E64}
    d_gP_Tf(h)=\E\left[f(\bbb^{g}_T)\left(-\sum\limits_{k\in J_K}\left\langle \xi_{3k},  \hat U_k\right\rangle\right)\right],
\end{equation}
where $(\hat U_k)_{k\ge 0}$ is given by~\eqref{E48}.
\end{thm}
We then deduce   reverse Poincaré inequalities.
\begin{thm}\label{T-RP}
With the same notation as in Theorem \ref{T-B}. For any $p\in(1,\infty]$, denoting $q\in [1,\infty)$ satisfying $\displaystyle \frac1{p}+\frac1{q}=1$,  we have 
\begin{equation}
    \label{E64ter}
    |d_gP_Tf(h)|\le\left(P_T|f|^p\right)^{1/p}m_q \ \E\left[\left(\sum\limits_{k\in J_K}\left\|\hat U_k\right\|^2_2\right)^{q/2}\right]^{1/q}.
\end{equation}
with $m_q^q=\E[|Z|^q]$ the $q$-th moment of a $\SN(0,1)$-variable $Z$.
The right hand side is finite for all $q\geq 1$ when $K=\infty$.

In the special case $p=q=2$,  we get the reverse Poincaré inequality
\begin{equation}
    \label{E64.4}
    \begin{split}
    &|d_gP_Tf(h)|^2\\&
    \le\left(P_T|f|^2\right)\left(
    \frac{\|h_x\|_2^2}
    {T}+
    \left(6\sqrt{2n}+\frac{4\sqrt{2}}{\sqrt{n}}\right)^2\left(\frac1{T^2}\|h_z-\frac{1}{2}x\symp h_x\|^2
    +\frac{2(n-1)}{3T}\|h_x\|_2^2
    \right)\right).
    \end{split}
\end{equation}

\end{thm}
\begin{proof}[Proof of Theorem \ref{T-B}]
Considering a vector $h=(h_x,h_z)\in \Ge_n$, we will compute 
\begin{equation}
    \label{E65}
    \lim_{a\to0}\frac1a\left(P_Tf(g+ah)-P_Tf(g)\right).
\end{equation}
Denote $\tilde g(a)=(\tilde x(a), \tilde z(a))=g+ah$. The matrix $W(\tilde g(a))$ defined in~\eqref{E31} rewrites as 
\begin{equation}
    \label{E66}
    W(\tilde g(a))=z-\tilde z(a) - \frac 1 2 \tilde x_a \symp x +(x-\tilde x(a))\symp \left(\frac{\sqrt T}2 \xi_0-\sqrt{T}\alpha_0\tilde \xi_1 \right)
\end{equation}
and since $x\symp x=0$,
\begin{align}
    \label{E67}
    \forall a\in \R, \quad d_{\tilde g(a)}W(h)&=\frac{d}{da}W(\tilde g(a))=
    -h_z- \frac 1 2 h_x\symp x -  h_x\symp \left(\frac{\sqrt T}2 \xi_0-\sqrt{T}\alpha_0\tilde \xi_1\right)\notag\\
    &=W(\tg(1)) \text{ not depending on $a$.}
\end{align}

 Consequently, with the notation of~\eqref{E52}, 
\begin{equation}
    \label{E68}
     d_{\tilde g(a)}\hat\SU^t(h)={-\frac{1}{2}}
   \hat\SV^t(\hat\SV\hat\SV^t)^{-1}d_{\tilde g(a)}\SW^t(h)={-\frac{1}{2}}\hat\SV^t(\hat\SV\hat\SV^t)^{-1}\SW^t(\tilde g(1))
\end{equation}
does not depend on $a$.

 Letting $\hat\SU^t=\hat\SU^t(\tilde g(1))$,  $\SW^t=\SW^t(\tilde g(1))$ and $(u_0,u_3,u_6,...)=\left(\hat{U}_0(\tg(1)),\hat{U}_1(\tg(1)),\hat{U}_2(\tg(1)),...\right)$,
\begin{equation}
    \label{E69}
     \hat\SU^t(\tilde g(a))=a\hat\SU^t=-\frac{a}{2}\hat\SV^t(\hat\SV\hat\SV^t)^{-1}\SW^t.
\end{equation}
also, $\hat U_0(\tilde g(a))=a\frac{-h_x}{\sqrt{T}}$ yielding $\frac{d}{da}\hat U_0(\tilde g(a))= u_0$.
Then using~\eqref{E56} and the fact that $\hat\SU=(u_3,u_6,\ldots)$ we get
\begin{align*}
    \frac1a\left(P_Tf(g+ah)-P_Tf(g)\right)&=\frac1a\E\left[f(\bbb^{g}_T)\left(R(au)-1\right)\right]\\
    &=\frac1a\E\left[f(\bbb^{g}_T)\left(\int_0^a \frac{d}{da'}R(a'u) \, da'\right)\right]\\
    &=-\frac1a\E\left[f(\bbb^{g}_T(\omega))\left(\int_0^a R(a'u)\langle\omega+a'u, u\rangle \, da'\right)\right]\\
\end{align*}
By definition of $R(a'u)$ we have as soon as $\omega\mapsto F(\omega)$ is $\P$-integrable, that $\omega\mapsto F(\omega+a'u)$ is $R(a'u)\P$-integrable and  
\begin{equation}\label{E70}
\E[R(a'u)F(\omega+a'u)]=\E[F(\omega)].
\end{equation}
In our situation $f$ is bounded and $\displaystyle \langle\omega, u\rangle=\left\langle\omega ,\frac{u}{\|u\|}\right\rangle\|u\|$ is $\PP$-integrable since, conditioned to $\SG$ $\left\langle\omega, \frac{u}{\|u\|}\right\rangle$ has law $\cN(0,1)$, $\|u\|\le \frac{\left\|h_x
\right\|_2}{\sqrt{T}}+\|\hat\SU\|$,
$$
\|\hat\SU\|= \sqrt{\tr\left(\hat\SU\hat\SU^t\right)}=\frac{1}{2}\sqrt{\tr\left(\SW^t(\hat\SV\hat\SV^t)^{-1}\SW\right)}\le \frac{1}{2} \|\SW\|\left(\tr\left(\left(\hat\SV\hat\SV^t\right)^{-1}\right)\right)^{1/2},
$$
$\SW$ is Gaussian and independent of $\hat{\SV}$ and 
\begin{itemize}
    \item 
if $K=\infty$ then by Equation~\eqref{E59} $\left(\tr\left(\left(\hat\SV\hat\SV^t\right)^{-1}\right)\right)^{1/2}$ has exponential moments,
\item
if $K<\infty$ then $\left(\tr\left(\left(\hat\SV\hat\SV^t\right)^{-1}\right)\right)^{1/2}\le \beta_{K}\left(\tr\left(\left(\SV\SV^t\right)^{-1}\right)\right)^{1/2}$ which is integrable by~\eqref{E20.8}, since we choose $K\ge n+2$.
\end{itemize}

So we can apply equality~\eqref{E70} after exchanging the orders of integration (which is allowed here for the same integrability reasons), and we get
\begin{align*}
    \frac1a\left(P_Tf(g+ah)-P_Tf(g)\right)&=-\frac1a\int_0^a\E\left[f(\bbb^{g}_T(\omega))\left( R(a'u)\left\langle\omega+a'u, u\right\rangle\right)\right] \, da'\\
    &=-\frac1a\int_0^a\E\left[f(\bbb^{g}_T(\omega-a'u))\left\langle\omega ,u\right\rangle\right] \, da'\\
    &=-\E\left[\left(\frac1a\int_0^af(\bbb^{g}_T(\omega-a'u))\, da'\right)\left\langle\omega ,u\right\rangle\right].
\end{align*}
Since $f$ is bounded and continuous, and a.s.  $\bbb^{g}_T(\omega-a'u)\to \bbb^{g}_T(\omega)$ as $a'\to 0$ we can use the dominated convergence theorem to obtain
\begin{equation}
    \label{E71}
    \lim_{a\to0}\frac1a\left(P_Tf(g+ah)-P_Tf(g)\right)=-\E\left[f(\bbb^{g}_T(\omega))\left\langle\omega ,u\right\rangle\right]
\end{equation}
which yields~\eqref{E64}. 
\end{proof}

\begin{proof}[Proof of Theorem \ref{T-RP}]
To establish~\eqref{E64ter} we first use Hölder inequality which yields 
\begin{equation}
    \label{E72}
   |d_gP_Tf(h)|\le \E\left[|f|^p(\bbb^{g}_T)\right]^{1/p}\E\left[\left|-
   \langle\omega,u\rangle\right|^q\right]^{1/q}.
\end{equation}
As in the proof of Corollary~\ref{C1}, conditioning with respect to~$\SG$ we get
\begin{align*}
    \E\left[\left|-\langle\omega,u\rangle
    \right|^q\right]&=
    \E\left[\E\left[\left|-\langle\omega,u\rangle
    \right|^q\Bigg|\SG\right]\right]\\
    &=\E\left[\|u\|^q
    m_q^q\right]
\end{align*}
with $m_q^q=\E[|Z|^q]$ the $q$-th moment of a $\SN(0,1)$-variable $Z$. In particular $\|u\|^2=\sum\limits_{k\in J_K}\left\|  \hat U_k\right\|_2^2$ which proves~\eqref{E64ter}. Notice that when $K=\infty$ the last term is finite thanks to Lemma~\ref{L4} which implies that all moments of $(\hat \SV\hat\SV^t)^{-1}$ are finite. 
Finally, to prove~\eqref{E64.4} we apply~\eqref{E64ter} with $K=2n+1$ which allows to use~\eqref{E42.4} and \eqref{E42.5} with $\tg=g+h$.
\end{proof}

The next corollary completes Theorem~\ref{T-RP} with a kind of weak inverse log-Sobolev inequality.
\begin{cor}
\label{C3}
With the same notation as in Theorem \ref{T-B},  we have for all $\delta>0$
and nonnegative continuous function $f$,
\begin{equation}
    \label{E64ter2}
    \left|d_gP_Tf(h)\right|\le\delta P_T\left(f\ln\left(\frac{f}{P_Tf(g)}\right)\right)(g) +\frac1{2\delta}\E\left[f(\bbb^{g}_T)\left(\sum\limits_{k\in J_K}\left\|\hat U_k\right\|^2\right)\right].
\end{equation}
In particular,
\begin{equation}
    \label{E64.42}
    \left|d_gP_Tf(h)\right|\le\sqrt{2 P_T\left(f\ln\left(\frac{f}{P_Tf(g)}\right)\right)(g) \E\left[f(\bbb^{g}_T)\left(\sum\limits_{k\in J_K}\left\|\hat U_k\right\|^2\right)\right]}.
\end{equation}
\end{cor}
\begin{proof}
Again we start with Equation~\eqref{E64}.
As already seen in Equation~\eqref{E68}, the random vectors $\hat U_k=\hat U_k(h)$, $k\in J_K$, depend linearly on $h$. Moreover the right-hand-side of~\eqref{E64ter2} and~\eqref{E64.42} is the same for $h$ and $-h$. Consequently, possibly changing $h$ into $-h$, it is enough to establish~\eqref{E64ter2} and~\eqref{E64.42} for $h$ satisfying $d_gP_Tf(h)\ge 0$, or equivalently to replace $\left|d_gP_Tf(h)\right|$ by $d_gP_Tf(h)$ in the left-hand-side. 

  Conditioning the right-hand-side of Equation~\eqref{E64} with respect to $\SG$ and using the Young inequality from e.g. Lemma 2.4. in \cite{ATW_Harnack}, we obtain
\begin{align*}
    d_gP_Tf(h)&=\E\left[f(\bbb^{g}_T)\left(-\sum\limits_{k\in J_K}\left\langle \xi_{3k},  \hat U_k\right\rangle\right)\right]\\
    &=\E\left[\E\left[f(\bbb^{g}_T)\left(-\sum\limits_{k\in J_K}\left\langle \xi_{3k},  \hat U_k\right\rangle\right)\Big\vert \SG\right]\right]\\
    &\le \E\left[\delta\E\left[f(\bbb^{g}_T)\ln\left(\frac{f(\bbb^{g}_T)}{\E\left[f(\bbb^{g}_T)|\SG\right]}\right)\Big\vert \SG\right]+
    \delta\E\left[f(\bbb^{g}_T)\Big\vert \SG\right]\ln\E\left[
    e^{-\frac1{\delta}\sum\limits_{k\in J_K}\left\langle \xi_{3k},  \hat U_k\right\rangle }
    \Big\vert \SG\right]
    \right].
\end{align*}
Now since conditioning with respect to $\SG$ transforms $-\sum\limits_{k\in J_K}\left\langle \xi_{3k},  \hat U_k\right\rangle$ into a centered Gaussian variable we get 
\[
\E\left[
    e^{-\frac1{\delta}\sum\limits_{k\in J_K}\left\langle \xi_{3k},  \hat U_k\right\rangle }
    \Big\vert \SG\right]=e^{\frac1{2\delta^2}\sum\limits_{k\in J_K}\left\|\hat U_k\right\|^2}
\]
which yields
\begin{equation}\label{E73}
\delta\E\left[f(\bbb^{g}_T)\Big\vert \SG\right]\ln\E\left[
    e^{-\frac1{\delta}\sum\limits_{k\in J_K}\left\langle \xi_{3k},  \hat U_k\right\rangle }
    \Big\vert \SG\right]=\frac1{2\delta}\E\left[f(\bbb^{g}_T)\left(\sum\limits_{k\in J_K}\left\|\hat U_k\right\|^2\right)\Big\vert\SG\right]
\end{equation}
since $\sum\limits_{k\in J_K}\left\|\hat U_k\right\|^2$ is $\SG$-measurable.
Also letting $Y=\E\left[f(\bbb^{g}_T)|\SG\right]$ and using by Jensen's inequality $\E[Y\ln Y]\ge \E[Y]\ln\E[Y] $ we get
\begin{equation}\label{E74}
\E\left[\E\left[f(\bbb^{g}_T)\ln\left(\frac{f(\bbb^{g}_T)}{\E\left[f(\bbb^{g}_T)|\SG\right]}\right)\Big\vert \SG\right]\right]\le P_T\left(f\ln\left(\frac{f}{\P_Tf(g)}\right)\right)(g).
\end{equation}
From~\eqref{E73} and~\eqref{E74} we get~\eqref{E64ter2}. 
Finally,~\eqref{E64.42} is obtained with
\[ \delta=\sqrt{\frac{\E\left[f(\bbb^{g}_T)\left(\sum\limits_{k\in J_K}\left\|\hat U_k\right\|^2\right)\right]}{2P_T\left(f\ln\left(\frac{f}{\P_Tf(g)}\right)\right)(g)}}.\]
\end{proof}

As a final corollary we  provide  estimates of the horizontal and vertical differential of the heat kernel $(g,h)\mapsto p_t(g,h)$ on $\Ge_n$.
\begin{cor}\label{Prop: HeatKernel}
    There exist three positive constants $K(n)$, $K_1(n)$ and $K_2(n)$ only depending on $n$ such that:
    \begin{align}\label{eq: deriveHeat}
    |d_gp_t(0,\cdot)(h)|&\leq t^{-\frac{n^2}{2}}e^{-\frac{K(n)}{t}d_{cc}(0,g)^2}
    \left(K_1(n)\frac{\|h_x\|_2}{\sqrt{t}}+K_2(n)\frac{\|h_z-\frac{1}{2}x\symp h_x\|}{t}\right).
\end{align}
\end{cor}

\begin{proof}
From \cite{V-SC-C}, see also \cite{bb-revpoinc}, there exist some positive constants $\tilde{K}(n)$ and $\tilde{K}_1 (n)$ depending on $n$ such that:
\begin{equation}\label{eq: estimateHeatKernel}
    p_t(g,h)\leq \frac{\tilde{K}_1(n)}{t^{\frac{n^2}{2}}}e^{-\frac{\tilde{K}(n)}{t}d_{cc}(g,h)^2}.
\end{equation}
Set $g,h\in\Ge_n$ then $p_t(0,g)=P_{\frac{t}{2}}(p_{\frac{t}{2}}(0,\cdot))(g)$. 
Using the reverse Poincaré inequality from Theorem \ref{T-RP} with $f=p_{\frac{t}{2}}(0,\cdot)$:
\begin{align}\label{eq: CauchySchwartz}
     |d_gp_t(0,\cdot)(h)|^2
     &\leq\esp\left[p_{\frac{t}{2}}(0,\bbb^g_{\frac{t}{2}})^2\right]\notag\\
     &\times\left(
    \frac{2\|h_x\|_2^2}
    {t}+
    \left(6\sqrt{2n}+\frac{4\sqrt{2}}{\sqrt{n}}\right)^2\left(\frac4{t^2}\|h_z-\frac{1}{2}x\symp h_x\|^2
    +\frac{4(n-1)}{3t}\|h_x\|_2^2
    \right)\right).
\end{align}

We now examine $\esp\left[p_{\frac{t}{2}}(0,\bbb^g_{\frac{t}{2}})^2\right]$. Using \eqref{eq: estimateHeatKernel}:
\begin{align}\label{eq: kernelpart}
    \esp\left[p_{\frac{t}{2}}(0,\bbb^g_{\frac{t}{2}})^2\right]&=\int_{\Ge_n}p_{\frac{t}{2}}(0,l)^2p_{\frac{t}{2}}(g,l)dl\notag\\
    &\leq\int_{\Ge_n}\tilde{K}_1(n)^{3}\left(\frac{t}{2}\right)^{-\frac{3n^2}{2}}e^{-\frac{2\tilde{K}(n)}{t}(2d_{cc}(0,l)^2+d_{cc}(g,l)^2)}dl\notag\\
    &\leq\int_{\Ge_n}\tilde{K}_1(n)^{3}\left(\frac{t}{2}\right)^{-\frac{3n^2}{2}}e^{-\frac{2\tilde{K}(n)}{t}\left(2d_{cc}(0,l)^2+\left(d_{cc}(g,0)-d_{cc}(0,l)\right)^2\right)}dl\notag\\
    &\leq \tilde{K}_1(n)^{3}\left(\frac{t}{2}\right)^{-\frac{3n^2}{2}}\int_{\Ge_n}e^{-\frac{2\tilde{K}(n)}{t}d_{cc}(0,l)^2}dl\,e^{-\frac{\tilde{K}(n)}{t}d_{cc}(0,g)^2}
\end{align}
where the last expression is obtained by using the inequality:
\begin{align*}
2a^2+(a-b)^2&=3a^2+b^2-2ab\geq 3a^2+b^2-\frac{a^2}{\lambda}-\lambda b^2\\
&=a^2+\frac{b^2}{2}\text{ with }\lambda=1/2.
\end{align*}

Using the property of the dilation on $(\Ge_n,d_{cc})$, $\frac{1}{\sqrt{t}}d_{cc}(0,l)=d_{cc}(0,\dil_{\frac{1}{\sqrt{t}}}(l))$, and since the homogeneous dimension of $\Ge_n$ is $n^2$, we have:
\begin{equation*}\left(\frac{t}{2}\right)^{-\frac{n^2}{2}}\int_{\Ge_n}e^{-\frac{2\tilde{K}(n)}{t}d_{cc}(0,l)^2}dl=\int_{\Ge_n}e^{-2\tilde{K}(n) d_{cc}(0,l)^2}dl\end{equation*} which  is finite and does not depend on $t$. The expected result follows. 
\end{proof}

\begin{NB}
If in the above proof,  one uses  the reverse Poincaré inequality \eqref{E64ter} with $p=1+\ve$ for $\ve >0$ (and with $K=+\infty)$, it is possible to obtain \eqref{eq: deriveHeat} with some constants $K(n,\ve)$, $K_1(n,\ve)$ and $K_2(n,\ve)$ depending only on $n$  and on $\ve$ with 
\[K(n, \ve)= \frac{\tilde K(n)}{1+\ve}\]
and where  $K_1(n,\ve)$ and $K_2(n,\ve)$ tend to infinity as $\ve \to 0$.
\end{NB}

\bibliographystyle{plain} 
\bibliography{Bibliographie}

\end{document}